\documentclass[a4paper, english, 10pt, onecolumn, oneside]{article}	
\usepackage[verbose, a4paper, hcentering]{geometry}	

\usepackage[dvipsnames]{xcolor}
\definecolor{greentuc}{rgb}{0.0, 0.38, 0.30}

\usepackage[utf8]{inputenc} 
\usepackage[T1]{fontenc}    
\usepackage[pdftitle={Unbalanced Sinkhorn}, pdfauthor={Rajmadan Lakshmanan}, citecolor=greentuc, backref=page, urlcolor= greentuc, linkcolor= greentuc, unicode=true,  bookmarks=true, bookmarksnumbered=false, bookmarksopen=false, breaklinks=false, pdfborder={0 0 0}, pdfborderstyle={}, colorlinks=true]{hyperref}
\usepackage[numbers]{natbib}	
\usepackage{siunitx} 
\usepackage{booktabs}       
\usepackage{amsthm, mathtools, amssymb, nicefrac}
\usepackage{amsmath}
\usepackage[english]{babel} 
\numberwithin{equation}{section}
\usepackage{unicodeMath}
\mathtoolsset{showonlyrefs}

\usepackage{microtype}      
\usepackage{newtxtext, newtxmath}
\usepackage{dsfont}   
\usepackage{tikz}
\usetikzlibrary{calc,trees,positioning,arrows,fit,shapes,calc}
\usepackage{pgfplots}
\usepackage{subfig}
\usepackage{blindtext}
\pgfplotsset{compat=newest}

\usepackage[inline, shortlabels]{enumitem}           
\setlist[enumerate,1]{label=(\roman*)}  
\newcommand{\one}{{\mathbf 1}} 		
\newcommand{\zero}{{\mathbf 0}} 		

\usepackage{xcolor}         
\usepackage{siunitx}
\usepackage{todonotes}
\usepackage[export]{adjustbox}
\usepackage{mathtools}
\usepackage{subfig}
\usepackage{graphicx}

\usetikzlibrary{trees,matrix,arrows.meta}
\usepackage[ruled, vlined]{algorithm2e}			
\usepackage{orcidlink}
\usepackage{diagbox}
\usepackage{doi}

\newcounter{relctr} 
\everydisplay\expandafter{\the\everydisplay\setcounter{relctr}{0}} 

\newcommand{\KL}{Ⅾ_{\mathrm{KL}}}
\newcommand{\uot}{\mathrm{UOT}}
\newcommand{\ot}{\mathrm{OT}}
\newcommand{\mmd}{\mathrm{MMD}}

\newcommand{\lhs}{\mathrm{LHS}}
\newcommand{\rhs}{\mathrm{RHS}}

\newcommand{\rev}[2][black]{{\color{#1}#2}}       
\newenvironment{revise}[1][black]{\color{#1}}{}   

\AtBeginDocument{}

\theoremstyle{plain}
	\newtheorem{theorem}{Theorem}[section]		
	\newtheorem{corollary}[theorem]{Corollary}
	\newtheorem{lemma}[theorem]{Lemma}
	\newtheorem{proposition}[theorem]{Proposition}
\theoremstyle{definition}
	\newtheorem{definition}[theorem]{Definition}
\theoremstyle{remark}
	\newtheorem{remark}[theorem]{Remark}

\DeclareMathOperator{\E}{\mathds E}	        

\title{Unbalanced Optimal Transport and Maximum Mean Discrepancies: Interconnections and Rapid Evaluation}
\author{%
	Rajmadan Lakshmanan%
		\thanks{Faculty of Mathematics, University of Technology, Chemnitz, Germany}\,
		\footnote{\orcidlink{0009-0006-3273-9063} \href{https://orcid.org/0009-0006-3273-9063}{https://orcid.org/0009-0006-3273-9063}. Contact: \protect\href{rajmadan.lakshmanan@math.tu-chemnitz.de}{rajmadan.lakshmanan@math.tu-chemnitz.de}}
	\and
	Alois Pichler%
		\footnotemark[1]\,
		\thanks{\href{https://orcid.org/0000-0001-8876-2429}{\orcidlink{0000-0001-8876-2429} https://orcid.org/0000-0001-8876-2429}; DFG, German Research Foundation – Project-ID 416228727 – SFB 1410.		
		}
}

\begin{document}
\maketitle

\begin{abstract}
This contribution presents substantial computational advancements to compare measures even with varying masses.
Specifically, we utilize the nonequispaced fast Fourier transform to accelerate the radial kernel convolution in unbalanced optimal transport approximation, built upon the Sinkhorn algorithm. We also present accelerated schemes for maximum mean discrepancies involving kernels.
Our approaches reduce the arithmetic operations needed to compute distances from $𝓞(n²)$ to $𝓞(n \log n)$, opening the door to handle large and high-dimensional datasets efficiently. 
\rev{Furthermore, we establish robust connections between transportation problems, encompassing Wasserstein distance and unbalanced optimal transport, and maximum mean discrepancies.
This empowers practitioners with compelling rationale to opt for adaptable distances.}
	\medskip


\end{abstract}

\section{Introduction}\label{sec:Intro}
Many progressive and remarkable formulations have been presented to address the problem of comparing probability measures.
Among popular domains like data science, the restriction to formulations that can only handle measures of equal weight (probability measures, for example), is cumbersome.
Some formulations have been proposed to handle unbalanced measures. 
However, the numerical acceleration of such formulations is not addressed enough in the literature.
This paper focuses on numerical acceleration of prominent formulations, which enable the comparison of two measures with possibly different masses.
Additionally, we provide theoretical bounds and elucidate relationships within the taxonomy of distances presented below.
\rev{Leveraging fast computations alongside these bounds empowers practitioners to select flexible distances tailored to their specific requirements.}

Optimal transport (OT) in its standard formulation builds on efficient ways to rearrange masses between two given probability distributions. 
Such approaches commonly relate to the Wasserstein or Monge–Kantorovich distance.
We refer to the monograph \citet{Villani2003} for an extensive discussion of the OT problem.
A pivotal prerequisite of the standard OT problem formulation is that it requires the input measures to be normalized to unit mass --~that is, to probability, or balanced measures.
This is an unfeasible presumption for some problems that need to handle arbitrary, though positive measures.

\emph{Unbalanced} optimal transport (UOT) has been established to deal with this drawback by allowing mass variation in the transportation problem (cf.\ \citet{benamou_2003}).
This problem is stated as a generalization of the Kantorovich formulation (cf.\ \citet{kantorovich1942transfer}) by taking aberrations from the desired measures into account in addition.

On the other hand, kernel techniques lead to many mathematical and computational advancements.
One popular instance of such advancements is the reproducing kernel Hilbert space (RKHS), which is predominantly considered as an efficient tool to generalize the linear statistical approaches to non-linear settings, cf.\ the \emph{kernel trick}, e.g.
By utilizing the distinct mathematical properties of RKHS, a distance measure is proposed in \citet{gretton2012kernel}, which is known as maximum mean discrepancy (MMD).
Maximum mean discrepancies (MMDs) are \emph{kernel} based distance measures between given \emph{unbalanced} and/\,or probability measures, based on embedding measures in a RKHS.
We refer to \citet{muandet2017kernel} for a textbook reference.
Remarkably, these kernel based distances provide meaningful metrics for \emph{unbalanced} measures as well.

\paragraph{Numerical computation and applications of UOT and MMD.}
Traditional choices for solving the transportation problem involve several discrete combinatorial algorithms that rely on the finite-dimensional linear programming formulation. Notable among these are the Hungarian method, the auction algorithm, and the network simplex, see  \citet{luenberger2016linear}, \citet{Ahuja1993}.
However, their scalability diminishes notably when confronted with large, dense problems. 
The technique of entropy regularization to solve the standard OT problem is an important milestone, which improves the scalability of traditional methods (cf.\ \citet{cuturi2013sinkhorn}).
The Sinkhorn algorithm exploits the entropy-regularized (OT) problem, recognized as an alternating optimization algorithm, commonly referred to as the \emph{Iterative Proportional Fitting Procedure} (IPFP) (cf.\ \citet{Sinkhorn1967a}).
The unbalanced OT (UOT) problems splendidly adopt the entropy regularization technique, which also improves the scalability.

Today’s data-driven world, which is dominated by rapidly growing machine learning (ML) techniques, utilizes the entropy regularized UOT algorithm for many applications.
These include, among various others, domain adaptation (cf.\ \citet{pmlr-v139-fatras21a}), crowd counting (cf.\ \citet{ma2021learning}), bioinformatics (cf.\ \citet{schiebinger2019optimal}), and natural language processing (cf.\ \citet{pmlr-v129-wang20c}).

MMD is considered to be an important framework for many problems in machine learning and statistics.
Intelligent fault diagnosis (cf.\ \citet{mmdML}), two-sample testing (cf.\ \citet{gretton2012kernel}), feature selection (cf.\ \citet{song2012feature}), density estimation (cf.\ \citet{song2008tailoring}), and kernel Bayes’ rule (cf.\ \citet{fukumizu2013kernel}) are applicable fields based on MMD.
In the MMD framework, the choice of the kernel plays a vital role, and it depends on the nature of the problem.
The forthcoming sections explicitly explain the efficient computational approach for prominent kernels.

\paragraph{Related works.}
Considering unbalanced measures (that is, measures of possibly different total mass) requires extending divergences to more general measures than probability measures.
We build our formulations on Bregman divergences.
This approach the popular Kullback–Leibler divergence, which captures deviations for probability measures only.

Many algorithms have been proposed, including the entropy regularization approach, to efficiently solve UOT problems.
\citet{chizat2018scaling} investigate some prominent numerical approaches. In the research work of \citet{carlier2017convergence}, a method for fast Sinkhorn iterations, that involve only Gaussian convolutions, is theoretically studied.
The core idea of this approach is that each step (i.e., each iteration) can be solved on a regular grid (equispaced), which is relatively faster than the standard Sinkhorn iteration. 
However, this approach utilizes the \emph{equispaced} convolution, which is often a setback among the wide range of applications (cf.\ \citet{platte2011impossibility}, \citet[Section~1]{potts2001fast}), with the expense of the approximation. 
Furthermore, specific results scrutinize the consideration of low-rank approximation methods, specifically Nyström methods, for enhancing scalability (cf.\ \citet{NEURIPS2019_f55cadb9}).
Moreover, \citet{von2023generalized} emphasize the necessity for addressing the heightened computational complexity and memory requirements of regularized UOT problems.

Despite powerful statistical properties, the MMD frameworks suffer from computational burdens.
Unlike OT and UOT problems, only few contributions are presented to surpass the computational burden of MMD problems.
Those few contributions improve the computational process at the price of substandard approximation accuracy (cf.\ \citet{zhao2015fastmmd}, \citet{le2013fastfood}).
Furthermore, some recent approaches, with focus on Gaussian kernel implementation, utilize the low-rank Nyström method to mitigate the computational burdens (cf.\ \citet{cherfaoui2022discrete}).

One of the top ten algorithms of the 20\textsuperscript{th} century is the fast Fourier transform (FFT), which relies on equispaced data points.
The technique of FFT has been generalized to access \emph{non}-equispaced data points, which is known as non-equispaced fast Fourier transform (\emph{N}FFT).
In contrast to FFT, NFFT is an approximate algorithm, although it renders stable computation at the same number of arithmetic operations as FFT.
The fast summation method employed in the forthcoming algorithms is harnessed in the standard entropy-regularized OT problem, see \citet{lakshmanan2023nonequispaced}.
Remarkably, it captures the precise nature of standard numerical algorithms and comes with significant improvements in terms of time and memory.
This method can also be utilized for multi-marginal OT, a generalization of standard OT (cf.\ \citet{Ba2022}).
\rev{Notably, the NFFT fast summation method has not been explored in either a UOT or an MMD setup.
A evident deviation regarding the utilization of NFFT approximation from existing literature lies in our implementation of the three-dimensional approximation regime of NFFT. 
Furthermore, the investigation into one, two, and three-dimensional approximations of MMD, notably focusing on the inverse multi-quadratic and energy kernels, is conspicuously absent in current literature.}

\paragraph{Contributions \rev{and outline}.}To address the formidable challenges set by UOT and MMDs, we introduce efficient methods and \rev{robust bounds.}
\begin{revise}
More precisely, the paper discusses the following key approaches to address unbalanced problems:
\begin{enumerate}
	\item We establish bounds of the UOT problem using the Wasserstein distance, suitable exclusively for probability measures, and the Bregman divergences, which are applicable for unbalanced measures.
	\item The relationship between MMD and the Wasserstein distance, as well as the optimal transport for unbalanced measures, recognizing that the genuine Wasserstein distance and UOT are often unattainable.
	\item NFFT-based fast computation methods for regularized UOT and MMD.
\end{enumerate}
\end{revise}

Initially, we conduct a comprehensive theoretical exploration, unveiling metric inequalities and other profound relationships between \rev{the Wasserstein distance and UOT}, and MMDs.
This paper addresses computational hurdles, but also deepens the understanding of the intricate connections within these frameworks.
\medskip

Below, we provide an outline of the paper.
\begin{enumerate}
	‣	Sections~\ref{sec:upperbound} and~\ref{sec:EntUOTConti} present upper bounds for both, the standard UOT problem and the regularized UOT problem. These bounds do not necessitate any sophisticated optimization routine for their computation.
	\rev{Moreover, Section~\ref{sec:upperbound} demonstrates a bound for UOT by leveraging the Wasserstein distance.} 
	‣	In Section~\ref{sec:Relation}, we furnish robust inequalities that articulate the relationship between MMD and the Wasserstein distance.
	\rev{Additionally, we expand the robust inequalities to the unbalanced setting, specifically examining the connection between MMD and UOT.}
	‣	 Section~\ref{sec:Fast_summation} introduces the fast summation technique based on the NFFT, enabling fast matrix-vector operations.
	This section unveils the NFFT-accelerated implementation of regularized UOT and MMDs, ensuring fast and stable computations.
	For fast MMD approaches, we present inverse multiquadratic and energy kernels in addition to Gaussian and Laplace kernels, as there is a demand for these kernels in the literature, see \citet{hagemann2023posterior}, \citet{energyAppl}, \citet{EnergyApp}.
	Furthermore, we elucidate the arithmetic operations associated with these problems.
	\rev{By leveraging our fast summation method and the bounds delineated in Section~\ref{sec:upperbound},~\ref{sec:EntUOTConti} and ~\ref{sec:Relation}, we adeptly acquire the desired quantities with efficiency and precision.}
	‣	Section~\ref{Sec:Numerical_exposition} substantiates the performances of our proposed approaches using both synthetic and real datasets.
	\rev{Furthermore, we provide an analysis of the disparities among our established bounds.} Additionally, we offer comprehensive discussions on the advantages and trade-offs associated with prominent existing approaches.
\end{enumerate}

\section{Preliminaries and taxonomy of Wasserstein distances} \label{sec:Preliminaries}
In this section, we provide necessary definitions and properties to facilitate the upcoming discussion.
Moreover, we introduce robust relations between MMD and Wasserstein distance, and upper bounds for the UOT and regularized UOT.

\subsection{Bregman divergence}
In various domains, the Bregman divergence (cf.\ \Citet{bregman1967relaxation}) is utilized as a generalized measure of difference of two different points (cf.\ \citet{lu2022neural, nielsen2022statistical}).
Here, we consider the deviation with regard to a strictly convex function on the set of non-negative Radon measures $𝓜 ₊(𝓧)$, where $𝓧 ⊂ ℝᵈ$. 
\begin{definition}[Bregman divergence]\label{def:226}
	For a $ℝ$-valued, convex function $Φ∶ 𝓜 ₊(𝓧)→ ℝ$, the Bregman divergence is
	\[	Ⅾ(ν‖μ)≔ Φ(ν)- F_μ(ν)- Φ(μ),\]
	where
	\[	F_μ(\nu)≔ \lim_{h↓0} \frac1h❨\Phi❪h \nu+ (1-h)μ❫- Φ(μ)❩ \]
	is the directional derivative of the convex function~$Φ$ at~$μ$ in direction $ν-μ$.
\end{definition}
In statistics, $F_μ$ is also called \emph{von Mises derivative} or the \emph{influence function} of $Φ$ at $μ$.
Note that the Bregman divergence exists (possibly with values $±∞$), and it is \emph{non-negative} (that is, $Ⅾ(ν‖μ)≥0$) \rev{even for unbalanced measures $\mu$ and $\nu$,} as the function $Φ$ is convex by assumption and $Φ(ν)≥ Φ(μ)+ F_μ(ν)$. 
Uniform convexity causes \emph{definiteness} (that is, $Ⅾ(ν‖μ)=0$ if and only if $ν=μ$) and the Bregman divergence, in general, is not symmetric.

The Bregman divergence is notably defined here for general measures $μ$ and $ν$ in the domain of $Φ$, including unbalanced measures (i.e., $μ(𝓧)≠ ν(𝓧)$); the Bregman divergence is particularly not restricted to probability measures.

Important examples for the Bregman divergence involve a non-negative reference measure~$μ$ and a convex function~$φ$, and build on
\[Φ(ν)≔ ∫_𝓧 φ❨Z_ν(ξ)❩ μ(ⅾξ),\]
where $Z_ν$ is the Radon–Nikodým derivative of $ν$ with respect to $μ$; that is, $ν(ⅾξ)= Z_ν(ξ) μ(ⅾξ)$.
As $φ$ is convex and the measure $μ$ positive, it follows from
\begin{align}
	Φ❨h ν₁+ (1-h)ν₀❩&= ∫_𝓧 φ⟮h {ⅾν₁∕ⅾμ}(ξ)+(1-h){ⅾν₀∕ⅾμ}(ξ)⟯ μ(ⅾξ) ⏎
		&≤ h∫_𝓧 φ⟮{ⅾν₁∕ⅾμ}(ξ)⟯μ(ⅾξ) + (1-h)∫_𝓧 φ ⟮{ⅾν₀∕ⅾμ}(ξ)⟯ μ(ⅾξ) ⏎
		&= h Φ(ν₁)+(1-h) Φ(ν₀)
\end{align}
that $Φ$ is convex as well.
For sufficiently smooth functions~$φ$, it follows with the Taylor series expansion $φ(1+z)= φ(1)+ φ′(1) z+𝓞(z²)$ that
\begin{align}
	F_μ(ν)&= \lim_{h↓0}{1∕h}❪∫_𝓧 φ⟮h⋅{ⅾν∕ⅾμ}(ξ)+1-h⟯ μ(ⅾξ)- ∫_𝓧 φ(1) μ(ⅾξ)❫  ⏎
		&= φ′(1)⋅ ∫_𝓧 ⟮{ⅾν∕ⅾμ}(ξ)-1⟯ μ(ⅾξ)⏎
		&= φ′(1) ν(𝓧)- φ′(1) μ(𝓧),
\end{align}
so that the Bregman divergence associated with~$Φ$ is
\begin{align}\label{eq:298}
	Ⅾ(ν‖μ)&= ∫_𝓧 φ⟮{ⅾν∕ⅾμ}(ξ)⟯ μ(ⅾξ)+ φ̕(1) ❨μ(𝓧)- ν(𝓧)❩ - φ(1) μ(𝓧).
\end{align}
Using the aforementioned universalized formulation of Bregman divergence, the Bregman divergence corresponding to $φ(z)= z\log z$ is
\begin{align}\label{eq:302}
	\KL(ν‖μ)&= ∫_𝓧 \log⟮{ⅾν∕ⅾμ}⟯\,ⅾν+ μ(𝓧)- ν(𝓧),
\end{align}
which generalizes the Kullback–Leibler divergence to general (unbalanced) measures~$ν$, provided that~$μ$ is a positive measure.

\subsection{OT problem}
The standard (i.e., balanced) OT problem is a linear optimization problem that reveals the minimal cost to transport masses from one probability measure to some other probability measure.
The optimal cost is popularly known as the Wasserstein distance.
\begin{definition}[$r$-Wasserstein distance]\label{def:Wasserstein}
	The Wasserstein distance of order $r≥1$ of the probability measures~$P$, $Ꝓ∈\mathcal P(\mathcal X)$ for a given cost or distance function $d∶𝓧× 𝓧→ [0,∞]$ is
	\begin{subequations}\label{eq:Wasserstein}
		\begin{align}\label{eq:W0}
			\rev{W_r(P,\tilde P)}& \coloneqq w_r(P,\tilde P)^{\nicefrac1r}, 	⏎
			\shortintertext{where}	
			w_r(P,\tilde P) & \coloneqq \inf_{\pi \in \mathcal P(\mathcal X × \mathcal X)}				\iint_{\mathcal X\times{\mathcal X}} d(x, 𝑥)ʳ  π(\mathrm d x,\mathrm d\tilde{x}).
		\end{align}
	Here, $π$ has marginal measures $P$ and $Ꝓ$, that is,
		\begin{align}
			τ_{1\#}π& = P	\text{ and} \label{eq:const_W} \\
			τ_{2\#}π& = Ꝓ,	\label{eq:const_Wb}
		\end{align}
	\end{subequations}
	where $τ₁(x,\tilde x)≔ x$ ($τ₂(x,\tilde x)≔ \tilde x$, resp.)\@ is the coordinate projection and $τ_{1\#}π≔ π∘τ₁^{-1}$ ($τ_{2\#}π≔ π∘τ₂^{-1}$, resp.)\@ the pushforward measure.
\end{definition}
More explicitly, the marginal constraints~\eqref{eq:const_W}--\eqref{eq:const_Wb} are
\begin{align}
	π(A×𝓧)& = P(A) \text{ and}⏎
	π(𝓧×B)& = Ꝓ(B),
	\end{align}
where $A$ and $B⊂ 𝓧$ are  measurable sets.
We shall also write
\[	⟨π| d^r⟩≔ ∬_{𝓧×𝓧} d(ξ,η)ʳ π(ⅾξ,ⅾη) \]
when averaging the function $dʳ$ with respect to the measure $π$ as in~\eqref{eq:W0}.

\subsection{UOT problem}
The marginal measures in~\eqref{eq:const_W} and~\eqref{eq:const_Wb} necessarily share the same mass, as $τ_{1\#} π(𝓧)= π(𝓧×𝓧)= τ_{2\#} π(𝓧)$.
The core principle of the UOT problem is to relax the hard marginal constraints~\eqref{eq:const_W} and~\eqref{eq:const_Wb} with soft constraints to enable the mass variation.
More concisely, the soft constraints are considered in the objective by involving the divergence between the marginals of~$π$ from given measures~$μ$ and~$ν$. 

In our research, we consider Bregman divergences as the soft constraints or \emph{marginal discrepancy} function.
\begin{definition}[Unbalanced optimal transport problem]\label{def:UOT_continuous}
	Let $μ$, $ν∈ 𝓜 ₊(𝓧)$ and $d∶𝓧×𝓧→ [0,∞)$. Let $Ⅾ(⋅ ‖ ⋅)$ be the Bregman divergence as in Definition~\ref{def:226}.
	The generalized unbalanced transport cost is
	\begin{equation}\label{eq:W3}
		\uot_{r;η}(μ,ν) ≔	\inf_{π∈ 𝓜 ₊(𝓧 × 𝓧)} ❬π| dʳ❭ 
			+ η₁ Ⅾ(τ_{1\#}π‖ μ)+ η₂ Ⅾ(τ_{2\#}π‖ ν),
	\end{equation}
	where $r≥ 1$; the parameter $η₁≥0$ ($η₂≥0$, resp.)\@ is a regularization parameter, which emphasizes the importance assigned to the marginal measures~$μ$ ($ν$, resp.).
\end{definition}
The measure $π∈𝓜 ₊$ in~\eqref{eq:W3} satisfies the relation
$τ_{1\#}π(𝓧)= τ_{2\#}π(𝓧)$, although $μ(𝓧) ≠ ν(𝓧)$ in general.
The optimal measure in~\eqref{eq:W3} thus constitutes a compromise between the total masses of~$μ$ and~$ν$.
Note as well that the measure $μ⊗ ν$ is feasible in~\eqref{eq:W3}, with total measure $(μ⊗ ν)(𝓧×𝓧)=μ(𝓧)⋅ν(𝓧)$.

\begin{remark}[Marginal regularization parameters]\label{rem:813}
	For parameters $η₁= η₂= 0$ (or $η₁= η₂↘ 0$, resp.), the explicit solution of problem~\eqref{eq:W3} is $π=0$ ($π↘0$, resp.).
	That is, no transportation takes place at all in this special case.

	For $η₁↗ ∞$ and $η₂↗ ∞$, problem~\eqref{eq:W3} appraises the marginals well.
	More precisely, for measures with equal mass, the solution~$π$ tends to the solution of the standard optimal transport problem.
	In addition, the total mass of the optimal solution~\eqref{eq:W3} increases from $π(𝓧×𝓧)=0$ to $π(𝓧×𝓧)= μ(𝓧)= ν(𝓧)$ (not $μ(𝓧)⋅ν(𝓧)$).
	For measures with unequal mass, $μ(𝓧)≠ ν(𝓧)$, the transportation plan is the best plausible arrangement of the measures~$μ$ and $ν$.
\end{remark}
For a discussion on the minimizing measure and its existence we may refer to \citet[Theorem~3.3]{liero2018optimal}.

\subsection{Upper bounds for the unbalance optimal transport \label{sec:upperbound}}
Suppose the marginals $τ_{1\#}π$ and $τ_{2\#}π$ were known, then the Wasserstein problem in~\eqref{eq:W3} is intrinsic.
For this reason, we may consider the Wasserstein problem with normalized marginals to develop an upper bound for the unbalanced optimal transport problem.
\begin{proposition}\label{prop:464}
	Let $π$ be a bivariate probability measure, feasible for the Wasserstein problem with normalized marginals
	\begin{equation}\label{eq:466}
		P(⋅)≔ {μ(⋅) ∕ μ(𝓧)}  \text{ and }  Ꝓ(⋅)≔ {ν(⋅) ∕ ν(𝓧)}.
	\end{equation}
	Then
	\begin{equation}\label{eq:478}
		c^*_{r;η}=e^{-{❬ π| dʳ❭ ∕ η₁+η₂}}⋅ μ(𝓧)^{η₁ ∕ η₁+η₂}⋅ν(𝓧)^{η₂ ∕ η₁+η₂}
	\end{equation}
	is an upper bound for the unbalanced optimal transport problem~\eqref{eq:W3} with Kullback–Leibler divergences,  and the constant in~\eqref{eq:478} is optimal among all measures of the form $c⋅π$, where $c≥0$.

\end{proposition}
\begin{proof}
	See Appendix~\ref{prop:464Proof}
\end{proof}

\begin{revise}
The following proposition delineates a bound for UOT by leveraging the Wasserstein distance and Bregman divergences. 
	
\begin{proposition}\label{prop:547}
	Let $μ$, $ν∈ 𝓜 ₊(𝓧)$, and define probability measures $P(⋅)≔ {μ(⋅) ∕ μ(𝓧)}$ and $Q(⋅)≔ {ν(⋅) ∕ ν(𝓧)}$.
	It holds that
	\[\uot_{r;η}(μ,ν) \leq u⋅ w_{r}(P,Q) + η_1 Ⅾ(P‖ μ)+ η_2 Ⅾ(Q‖ ν),  \]where  $u\coloneqq μ(𝓧)⋅ν(𝓧)$ and $Ⅾ(⋅‖⋅)$ is the Bregman divergence.

	For probability measures $P$ and $Q$ on a discrete space $\mathcal X$, it holds in addition that
	\[w_{r}(P,Q) \cdot\uot_{r;η}(P,Q)≤ \|d^r\|_F^2 ,\]
	where $\|d^r\|_F$ is the Frobenius norm of the distance matrix. 
\end{proposition}
\begin{proof} For probability measures $P$ and $Q$ as defined above, we have that
	\begin{align}\uot_{r;η}(μ,ν) &=  \inf_{{π}∈  𝓜 ₊(𝓧 × 𝓧)} ∬_{𝓧×𝓧} d(x,y)^{r}  π(ⅾx,ⅾy)  + η_1 Ⅾ(τ_{1\#}π‖ μ)+  η_2 Ⅾ(τ_{2\#}π‖ ν)	⏎ 
						&≤  \inf_{{π}∈  𝓜 ₊(𝓧 × 𝓧)}u⋅ ⟮∬_{𝓧×𝓧} d(x,y)^{r} {{π}(ⅾx,ⅾy) ⁄ u}+ \tilde{η}_1 Ⅾ⟮{τ_{1\#}{π} ⁄ u}‖ P⟯+  \tilde{η}_2 Ⅾ⟮{τ_{2\#}{π} ⁄ u}‖ Q⟯ ⟯		
		⏎ &\qquad \qquad\qquad\qquad \qquad + η_1 Ⅾ(τ_{1\#}π‖ μ)+ η_2 Ⅾ(τ_{2\#}π‖ ν)  \label{eq:766}\end{align}
	as the terms $u⋅\tilde{η}_1 Ⅾ({τ_{1\#}{π} ⁄ u}‖ P)$ and  $u⋅\tilde{η}_2 Ⅾ({τ_{2\#}{π} ⁄ u}‖ Q)$ are non-negative for $\tilde\eta_1\ge 0$ and $\tilde\eta_2\ge0$.
	Employing the optimal measure~$π^{*}_{\tilde{η}}$ of the UOT problem associated with the probability measures $P$ and $Q$, it follows from right-hand side of~\eqref{eq:766} that
		\begin{align} \uot_{r;η}(μ,ν)  ≤ u⋅ \uot_{r;\tilde{η}}(P,Q) + η_1 Ⅾ(τ_{1\#}π^{*}_{\tilde{η}}‖ μ)+ η_2 Ⅾ(τ_{2\#}π^{*}_{\tilde{η}}‖ ν). \label{eq:770}\end{align}
	Taking the limit on the right-hand side, i.e., $\tilde{η}_1,\tilde{η}_2\to \infty$, we obtain that~(cf.\ Remark~\ref{rem:813})
		\begin{align}\label{eq:567}
			\uot_{r;η}(μ,ν)  ≤ u⋅ w_{r}(P,Q) + η_1 Ⅾ(P‖ μ)+ η_2 Ⅾ(Q‖ ν),
		\end{align} which is the first assertion.
			Now, by exploiting Cauchy–Schwarz inequality, we have  
		\begin{align}\label{eq:572}
			 w_{r}(P,Q) ≤ \|d^r\|_F ⋅\|\tilde{\pi}\|_F ≤ \|d^r\|_F,
		\end{align} where $\tilde{\pi}$ is the optimal measures of $w_{r}(P,Q)$, and  $\|\tilde{\pi}\|_F≤ 1$ by Hölder’s inequality.
		Then, it follows from~\eqref{eq:567} and~\eqref{eq:572} for probability measures $P$ and $Q$ that
		\[
		{w_{r}(P,Q)⁄ \|d^r\|_F  }\le {w_{r}(P,Q)⁄ w_{r}(P,Q)} ≤ {w_{r}(P,Q)⁄ \uot_{r;η}(P,Q)}
		≤ \frac{\|d^r\|_F }{\uot_{r;η}(P,Q)},	⏎ \]
	  and thus the second assertion.
	\end{proof}

In the following Section~\ref{sec:EntUOTConti}, we discuss the regularized UOT problem, and its properties.  
\end{revise}
\subsection{Entropy regularized UOT\label{sec:EntUOTConti}}
The generic setup of UOT is computationally challenging.
Thus, an entropy regularization approach was proposed, enabling computational efficiency with negligible compromise in accuracy.
\begin{definition}\label{def:526}
	The entropy regularized, unbalanced optimal transport problem is
	\begin{equation}\label{eq:W4}
		{\uot}_{r;η;λ}(μ,ν) ≔ \inf_{π∈ 𝓜 ₊(𝓧 × 𝓧)} ❬π| dʳ❭
			+ {1 ∕ λ} Ⅾ(π‖ μ⊗ ν)
			+ η₁ Ⅾ(τ_{1\#}π‖ μ)+ η₂ Ⅾ(τ_{2\#}π‖ ν),
	\end{equation}
	where $π$ is a non-negative, bivariate measure on $𝓧×𝓧$; the parameter $r≥1$ is the order of the Wasserstein distance, $λ>0$ accounts for the entropy regularization and $η₁ ≥ 0$, $η_2≥ 0$ are the marginal regularization parameters.
\end{definition}
\begin{remark}\label{rem:541}
	It is notable that the constants in~\eqref{eq:W4} regularizing the marginals are $η₁$ and $η₂$, while the constant for the discrepancy of the entropy is $\nicefrac1{λ}$ (and not $λ$). We maintain the constant~$1∕λ$ instead of~$λ$ to stay consistent with preceding literature.
	To recover the initial Wasserstein problem~\eqref{eq:Wasserstein}, it is essential that $η₁$ and $η₂$, as well as $λ$ are large or tend to $+∞$.
\end{remark}
As in Proposition~\ref{prop:464} above, we can establish the following optimal compromise between the measures $μ$ and $ν$.
\begin{proposition}[Optimal independent measure]\label{prop:536}
	When restricted to measures $c⋅μ⊗ ν$, $c≥0$, the measure
	\[	π⃰ ≔ c^*_{r;η;λ} ⋅ μ⊗ ν \]
	with constant
	\begin{equation}\label{eq:561}
		c^*_{r;η;λ}≔ e^{-{❬μ⊗ ν| dʳ❭ ∕ μ(𝓧)ν(𝓧)(η₁+η₂+1/λ)}} 
		μ(𝓧)^{-{η₂ ∕ η₁+η₂+1/λ}} 
		ν(𝓧)^{-{η₁ ∕ η₁+η₂+1/λ}}
	\end{equation}
	is optimal for the entropy (Kullback–Leibler) regularized, unbalance optimal transport problem~\eqref{eq:W4} with Kullback–Leibler divergences.
\end{proposition}
\begin{proof}
	See Appendix~\ref{prop:536_proof}
\end{proof}
\begin{revise}
\begin{remark}[Upper bounds]\label{rem:upperbound}
	The determination of optimal quantities $c^*_{r;η;λ}$ and $c^*_{r;η}$ doesn't necessitate intricate optimization algorithms. 
	Furthermore, these upper bounds,  
	\begin{align}
		\uot_{r;η}(μ,ν) ≤ c^*_{r;η}   \text{and}  \uot_{r;η;λ}(μ,ν) ≤ c^*_{r;η;λ},
	\end{align}
	possess significant advantages across numerous scenarios, serving as robust initial values for iterative procedures in practical implementations.
	Through the utilization of these optimal quantities, we establish significant and non-trivial connections within a particular class of distances in the following Section~\ref{sec:Relation}.
\end{remark}
\end{revise}

\begin{remark}[Sinkhorn divergence]
 The addition of regularization term accelerates the computational process of both, OT and UOT problems. However, the transportation costs or distances of entropy regularized OT and UOT problems violate the axiom of definiteness of distances.

 The technique of entropy debiasing was initially proposed for the regularized OT problem by \citet{ramdas2017wasserstein}, which is formulated as
 \begin{equation}\label{eq:sd_OT}
	\mathrm{sd}_{r;λ}(P,\tilde{P}) ≔  w_{r;λ}(P,Ꝓ) - \frac12\,{w}_{r;λ}(P,P) - {1⁄ 2}\, w_{r;λ}(\tilde{P},\tilde{P}),
 \end{equation}
 where $w_{r;λ}(P,Ꝓ) ≔ \min_{\pi \in 𝓟(\mathcal X × \mathcal X)}\langle d^r,\pi\rangle + \frac1\lambda Ⅾ\big(\pi‖ P \otimes \tilde{P}\big)$.
 Here, $π$ has marginal measures~$P$ and~$Ꝓ$ as defined in~\eqref{eq:W0}, $r≥1$, and $λ>0$ is an entropy regularization parameter. 

 Similarly, to evacuate the bias, one can state the debiased regularized UOT problem as
 \begin{align}\label{eq:586}
	\mathrm{sd}_{r;η;λ}(μ,ν)
 		≔	& \uot_{r;η;λ}(μ,ν) - {1∕2} \uot_{r;η;λ}(μ,μ) - {1∕2} \uot_{r;η;λ}(ν,ν) ⏎
 		  & + {1∕ 2 λ}❨μ(𝓧) -ν(𝓧)❩².
 \end{align}
This technique of entropy debiasing for regularized OT and UOT problems, is adapted from the idea of computation of \emph{kernel} based distance, i.e., MMD. The aforementioned formulations are called Sinkhorn divergence (cf.\ \citet{genevay2018learning}).
\end{remark}
\begin{revise}
\begin{remark}[entropy upper bounds]
	For probability measures $P$ and $\tilde P$, as a natural consequence of the non-negative entropy penalty term $Ⅾ(⋅‖⋅)≥0$, we have $w_r(P,\tilde P) ≤ w_{r;λ}(P,\tilde P)$.
	Similarly, for the unbalanced case, we have $\uot_{r;η}(μ,ν) ≤ \uot_{r;η;λ}(μ,ν)$.
\end{remark}
\end{revise}
\subsection{Maximum mean discrepancies}\label{Sec:MMD}
Maximum mean discrepancy~(MMD) embeds measures into a reproducing kernel Hilbert space (RKHS).
This embedding gives rise to defining a distance on general, unbalance measures.

In what follows, we provide conditions so that the embedding is continuous with respect to the Wasserstein distance.
As well, we demonstrate the fast computation of the new distance on unbalanced measures.


\begin{definition}[Positive semi-definite]
	The function $k∶ 𝓧 × 𝓧 → ℝ$ is a \emph{positive definite} and symmetric kernel, if
	\[	∑_{i=1}ⁿ∑_{j=1}^n cᵢ cⱼ k(xᵢ,xⱼ) ≥ 0 \]
	for any scalars $(c₁,‥,cₙ)∈ ℝⁿ$ and $(x₁,‥,xₙ)∈ 𝓧ⁿ$ and $n∈ ℕ$.
\end{definition}
\begin{definition}[Reproducing kernel Hilbert space, cf.\ \citet{Berlinet2004}]
	A Hilbert space $𝓗$ of functions from $𝓧→ ℝ$ (with inner product $❬⋅,⋅❭_𝓗$) is a RKHS, if there is a positive semi-definite kernel $k∶ 𝓧×𝓧→ ℝ$, 
	such that
	\begin{enumerate}
		\item $kₓ(⋅)≔  k(⋅,x)∈ 𝓗$	for all $x∈ 𝓧$, and
		\item $❬kₓ, φ❭_𝓗 = φ(x)$	for all $φ∈ 𝓗$ and $x∈ 𝓧$ (that is, the evaluation of the function $φ∈𝓗$ is a continuous, linear functional).
	\end{enumerate}
	$k$ is called the \emph{kernel} of $𝓗$ (cf.\ \citet{aronszajn1950theory}).
\end{definition}
\begin{definition}[MMD]
	Let $𝓗ₖ$ be an RKHS with kernel $k$.
	The mean embedding is
	\begin{align}
		ι∶𝓜 &↪ 𝓗ₖ						\label{eq:660}⏎
		μ&↦ μₖ(⋅)≔ ∫_𝓧 k(x,⋅) μ(ⅾx).	\label{eq:659}
	\end{align}
	The RKHS distance between the mean embedding of given measures $μ$ and $ν$ is defined as 
	\begin{equation}\label{eq:711}\mmd_k(μ,ν)≔ ‖μₖ - νₖ‖_{𝓗ₖ}.\end{equation}
\end{definition}
This RKHS distance is popularly known as MMD.
Notably, it satisfies the axioms of a distance function.
Further, given the elementary function $f(⋅)= ∑_{j=1}ⁿ wⱼ k(xⱼ,⋅)∈ 𝓗ₖ$,
then $ι(μ)= f$ for the discrete measure $μ(⋅)= ∑_{j=1}ⁿ wⱼ δ_{xⱼ}$.

In other words, an RKHS-based interface between kernel methods and given distributions is offered by the mean embedding kernels. This distance is also known as \emph{mean map kernel}.

The critical relation to Wasserstein (Kantorovich) distances is given by
\[	∫_𝓧 φ ⅾμ=∫_𝓧 φ(x) μ(ⅾx)= ∫_𝓧 ❬φ(⋅)| k(⋅,x)❭μ(ⅾx)= ❬φ(⋅)\Big| ∫_𝓧 k(⋅,x) μ(ⅾx)❭= ❬φ| μₖ❭, \]
so that
\[	∫_𝓧 φ ⅾμ- ∫_𝓧 φ ⅾν = ❬μₖ-νₖ| φ❭≤ ‖μₖ-νₖ‖_{𝓗ₖ} ‖φ‖_{𝓗ₖ}\]
and
\begin{equation}\label{eq:670}
	‖μₖ-νₖ‖_{𝓗ₖ}= \sup_{φ∈ Bₖ} ∫_𝓧 φ ⅾμ-∫_𝓧 φ ⅾν,
\end{equation}
where the supremum is among all functions $φ∈ 𝓗ₖ$ in the unit ball $Bₖ⊂𝓗ₖ$ ($‖φ‖_{𝓗ₖ}≤ 1$).
By the Kantorovich--Rubinstein theorem (cf.\ \citet{RachevRueschendorf}), the equation~\eqref{eq:670} can serve as a definition of the Wasserstein distance of order $r=1$, where the functions $φ$ are among all Lipschitz‑1 functions.
\begin{revise}
\begin{remark}[MMD upperbound]\label{rem:MMDbound}
	In general, it holds that \begin{align}\label{eq:731}  ‖μₖ‖^2_{𝓗ₖ} = ∬_{𝓧×𝓧} k(x,x^\prime) μ(ⅾx)μ(ⅾx^\prime) \leq  \|k\|_∞ ⋅ μ(𝓧)^2.\end{align}
As a consequence of the triangle inequality and~\eqref{eq:731}, we have 
\begin{align} 
	\mmd_k(μ,ν)^2 = ‖μₖ - νₖ‖^2_{𝓗ₖ} ≤ ❪‖μₖ‖_{𝓗ₖ} + ‖νₖ‖_{𝓗ₖ}❫^2 ≤ \|k\|_∞ \cdot ❪μ(𝓧)^2 + ν(𝓧)^2❫,
\end{align}
which is the elementary upper bound.
\end{remark}
\end{revise}
\subsection{Relations between MMD and transportation problems}\label{sec:Relation}
In what follows we establish intricate connections between MMD and transportation problems, encompassing both probability and non-probability measures. 
To this end, consider first the measure $P= δₓ$ ($Q=δ_y$, resp.), the Dirac measure located at $x∈ 𝓧$ ($y∈ 𝓧$, resp.).
Assuming that the kernel $k$ is bounded by $C$ ($k≤C$), say, then
\begin{align}
		& wᵣ(δ_x,δ_y)= d(x,y),
\shortintertext{ while }
		&‖δ_{x,k}-δ_{y,k}‖²_{𝓗ₖ}= ‖k(⋅,x)- k(⋅,y)‖_{𝓗ₖ}²= k(x,x)-2k(x,y)+k(y,y)≤ 2C.
\end{align}
It follows that the transportation problems, in general, cannot be bounded by MMD on support sets~$𝓧$ with unbounded diameter $\sup_{x,y∈𝓧}d(x,y)$.
Note as well that $d(x,y)≔ ‖δₓ-δ_y‖_{𝓗ₖ}= √{k(x,x)-2k(x,y)+k(y,y)}$ defines a pseudo-metric on $𝓧$.

The following results provide conditions so that the embedding $ι∶𝓜 ↪ 𝓗ₖ$ is (Hölder‑)cont\-inuous.
It generalizes~\citet[Propostion~2]{VayerGribonval} slightly to kernels, for which the unit ball in~$𝓗ₖ$ is not necessarily uniformly Lipschitz, which includes the Laplacian kernel, e.g.

\rev{The subsequent Lemma~\ref{eq:681} expounds upon essential findings crucial for establishing the interrelationship between MMD and transportation problems.}
\begin{revise}

\begin{lemma}\label{eq:681}
Let $μ$, $ν ∈ 𝓜 ₊(𝓧)$, and define probability measures $P(⋅)≔ {μ(⋅) ∕ μ(𝓧)}$ and $Q(⋅)≔ {ν(⋅) ∕ ν(𝓧)}$. Additionally, consider functions $φ∈ 𝓗ₖ$ in the unit ball $Bₖ⊂𝓗ₖ$.
It holds that
\begin{equation}❪𝔼_P φ- 𝔼_Q φ❫²≤ ∬_{𝓧×𝓧} k(x,x)- 2k(x,y)+ k(y,y)\tilde{π}(ⅾx,ⅾy),\label{eq:678}
\end{equation} where $\tilde{π}$ has marginals $P$ and $Q$.
Furthermore, for unbalanced measures $μ$ and $ν$ with $η_1  > 0$, $η_2 > 0$ and $α>0$, $c>0$, it holds that  
\begin{align}
  ❪∬_{𝓧×𝓧} φ(x)-φ(y) {π}(ⅾx,ⅾy)❫² ≤ u^*\,⋅\, &∬_{𝓧×𝓧} k(x,x)- 2k(x,y)+ k(y,y) π(ⅾx,ⅾy) 	⏎&\quad + η_1 Ⅾ(τ_{1\#}π‖ μ)+ η_2 Ⅾ(τ_{2\#}π‖ ν); \label{eq:679} \end{align}
  here, $π$ represents the optimal measure of the problem $\uot_{2α;\nicefrac{η}{c^2}}(μ,ν)$ and $u^*= \pi(\mathcal X\times \mathcal X)$.
\end{lemma}
\begin{proof}
	First, assume that $‖φ‖_{𝓗ₖ}≤ 1$, then
	\begin{align}
		❨φ(x)-φ(y)❩²	
			&= ❬φ(⋅)| k(⋅,x)-k(⋅,y)❭² ⏎
			&≤ ‖φ‖_{𝓗ₖ}²⋅‖k(⋅,x)-k(⋅,y)‖_{𝓗ₖ}² ⏎
			&≤ k(x,x)- 2k(x,y)+ k(y,y).\label{eq:700}
	\end{align}
	Then, it follows with Jensen’s inequality that
	\begin{align}
		|𝔼_P φ- 𝔼_Q φ|²	&= ❪∬_{𝓧×𝓧} φ(x)-φ(y) \tilde\pi (ⅾx,ⅾy)❫² \\
						&≤ ∬_{𝓧×𝓧}❨φ(x)-φ(y)❩² \tilde\pi(ⅾx,ⅾy), \label{eq:17}
	\end{align}
	 where measure $\tilde \pi$ has marginals \(P\) and \(Q\).
	By substituting~\eqref{eq:700} in~\eqref{eq:17}, we obtain first assertion.
	
	Now, similar to first assertion, we have, with Jensen’s inequality
	\begin{align}
		❪∬_{𝓧×𝓧} φ(x)-φ(y) {π}(ⅾx,ⅾy)❫²	&= {u^*}^2❪∬_{𝓧×𝓧} φ(x)-φ(y) {{π}(ⅾx,ⅾy)⁄ u^*}❫² 	⏎
						&≤ u^*⋅∬_{𝓧×𝓧}❨φ(x)-φ(y)❩² {π}(ⅾx,ⅾy), 	⏎
						&≤ u^*⋅∬_{𝓧×𝓧}k(x,x)- 2k(x,y)+ k(y,y) {π}(ⅾx,ⅾy), \label{eq:692}
	\end{align}for the optimal measure $π$ as defined above and $u^\ast=\pi(\mathcal X\times\mathcal X)$. This quantity can be expressed with the marginals $\mu^\ast\coloneqq \tau_{1\#}\pi$ and $\nu^\ast\coloneqq \tau_{1\#}\pi$, it holds that $u^*=\sqrt{μ^*(𝓧)⋅ν^*(𝓧)}$.
	Finally, we obtain  
	\begin{align}❪∬_{𝓧×𝓧} φ(x)-φ(y) π(ⅾx,ⅾy)❫² ≤ u^*&\,⋅\, ⟮∬_{𝓧×𝓧} k(x,x)- 2k(x,y)+ k(y,y) π(ⅾx,ⅾy) 	⏎
		&+ η_1 Ⅾ(τ_{1\#}π‖ μ)+ η_2 Ⅾ(τ_{2\#}π‖ ν)⟯,  \end{align}
	as $Ⅾ(τ_{1\#}π‖ μ)≥0$ and $Ⅾ(τ_{2\#}π‖ ν)≥0$,  which is the second assertion. This completes the proof.
\end{proof}
\end{revise}
In the following result, we establish Hölder continuity between MMD distance for probability measures and Wasserstein distance.
\begin{proposition}\label{thm:666}
	Let $P$ and $Q$ be probability measures.
	Suppose that
	\begin{equation}\label{eq:686}
		k(x,x)-2k(x,y)+k(y,y)≤ c² d(x,y)^{2α}, x,y∈𝓧,
	\end{equation}
	for some $c>0$ and $α≥ 1/2$.
	Then it holds that
	\begin{equation}\label{eq:701}
		\mmd_k(P,Q)≤ c w_{2α}(P,Q)^{α}.
	\end{equation}
\end{proposition}
\begin{proof}
	It follows with \eqref{eq:678} that
	\begin{align}
		|𝔼_Pφ- 𝔼_Qφ|² &≤ ∬_{𝓧×𝓧}k(x,x)-2k(x,y)-k(y,y) π(ⅾx,ⅾy)⏎
					  &≤ ∬_{𝓧×𝓧}c² d(x,y)^{2α} π(ⅾx,ⅾy)	\label{eq:702},
	\end{align}
	where we have used~\eqref{eq:686}. Maximizing the left-hand side with respect to all functions $φ∈ 𝓗ₖ$ with $‖φ‖_{𝓗ₖ}≤1$ using~\eqref{eq:670} and minimizing its right-hand side with respect to all measures~$π$ with marginals~$P$ and~$Q$ reveals that
	\[	 \mmd_k(P,Q)^2 =‖Pₖ - Qₖ‖_{𝓗ₖ}² ≤ c² w_{2α}(P,Q)^{2α}.\]
	Hence, the assertion.
\end{proof}
\begin{revise}
The Wasserstein distance imposes fixed marginals in~\eqref{eq:const_W} and~\eqref{eq:const_Wb}.
In contrast, the unbalanced problem relaxes these constraints by incorporating soft marginal constraints in~\eqref{eq:W3}. 
Furthermore, the maximum mean discrepancy (cf.~\eqref{eq:711}) accommodates even for unbalanced measures.

The subsequent result establishes a connection between the MMD distance and UOT for general, unbalanced measures.
The quantities $\mmd_k(\mu,\mu^\ast)$ and $\mmd_k(\nu,\nu^\ast)$ in~\eqref{eq:704} below account for the varying masses of the measures.
These quantities vanish as $η₁↗ ∞$ and $η₂↗ ∞$, and if $\mu$ ($\nu$, resp.) are probability or balanced measures.
\begin{proposition}\label{thm:863}
	Suppose that
	\begin{equation}\label{eq:763}
		k(x,x)-2k(x,y)+k(y,y)≤ c² d(x,y)^{2α}, x,y∈𝓧,
	\end{equation}
	for some $c>0$ and $α≥ 1/2$.
	Let $μ$, $ν ∈ 𝓜 ₊(𝓧),$ and $π$ be the optimal measure of the problem $\uot_{2α;\nicefrac{η}{c^2}}(μ,ν)$ with marginals $\mu^\ast\coloneqq τ_{1\#}{π}$ and $\nu^\ast\coloneqq τ_{2\#}{π}$.
	Then, it holds that
	\begin{equation}\label{eq:704}
		\mmd_k(μ,ν) ≤ c⋅ \sqrt{u^*}⋅\uot_{2α;\nicefrac{η}{c²} }(μ,ν)^{α} + \mmd_k(μ,μ^*)+ \mmd_k(ν,ν^*),
	\end{equation}
	where $u^*=\sqrt{μ^*(𝓧)⋅ν^*(𝓧)}$. 
	
\end{proposition}
\begin{proof}
	Consider the optimal measure $π$ of the problem $\uot_{2α;\nicefrac{η}{c^2}}(μ,ν)$, which has marginals $μ^*$ and $ν^*$ and note that $\mu^\ast(\mathcal X)=\nu^\ast(\mathcal X)$.
	Utilizing~\eqref{eq:679} and~\eqref{eq:763}, we have 
	\begin{align}
		❪∬_{𝓧×𝓧} φ(x)-φ(y) {π}(ⅾx,ⅾy)❫² &≤ u^*⋅⟮∬_{𝓧×𝓧} k(x,x)- 2k(x,y)+ k(y,y) π(ⅾx,ⅾy)	⏎ & \qquad\qquad + η_1 Ⅾ(τ_{1\#}π‖ μ)+ η_2 Ⅾ(τ_{2\#}π‖ ν)⟯⏎
					  &≤ u^*⋅⟮∬_{𝓧×𝓧} c² d(x,y)^{2α} π(ⅾx,ⅾy)+ η_1 Ⅾ(τ_{1\#}π‖ μ)	⏎ & \qquad + η_2 Ⅾ(τ_{2\#}π‖ ν)⟯	⏎
					  &= c²⋅u^*⋅⟮\inf_{\bar{π}∈ 𝓜 ₊(𝓧 × 𝓧)} ∬_{𝓧×𝓧}  d(x,y)^{2α} \bar{π}(ⅾx,ⅾy)+ {η_1⁄c²} Ⅾ(τ_{1\#}\bar{π}‖ μ)
					  ⏎ & \qquad + {η_2⁄c²} Ⅾ(τ_{2\#}\bar{π}‖ ν)⟯
					  \label{eq:784}.	\end{align}
	Maximizing the left-hand side with respect to all functions $φ∈ 𝓗ₖ$ with $‖φ‖_{𝓗ₖ}≤1$ using~\eqref{eq:670}, we obtain 
	\begin{equation}\label{eq:828}\mmd_k(μ^*,ν^*)^2= ‖μ^*_k- ν^*_k‖^2_{𝓗ₖ} ≤ c²⋅ u^*⋅\uot_{2α;\nicefrac{η}{c²} }(μ,ν)^{2α}.\end{equation}
	By adding $\mmd_k(ν,ν^*)$ on both side of \eqref{eq:828} and utilizing the triangle inequality (that is, $\mmd_k(μ^*,ν) ≤ \mmd_k(μ^*,ν^*)+ \mmd_k(ν,ν^*)$) on the left-hand side to obtain
	\begin{align}\label{eq:843}	 
		\mmd_k(μ^*,ν)   &≤  c⋅ \sqrt{u^*}⋅\uot_{2α;\nicefrac{η}{c²} }(μ,ν)^{α} +\mmd_k(ν,ν^*).
	\end{align} 
	Repeating the same procedure by adding $\mmd_k(μ,μ^*)$ on both side of \eqref{eq:843}, we obtain  
	\begin{align}\label{eq:846}	 
		\mmd_k(μ,ν)   &≤  c⋅ \sqrt{u^*}⋅\uot_{2α;\nicefrac{η}{c²} }(μ,ν)^{α} +\mmd_k(ν,ν^*) + \mmd_k(μ,μ^*),
	\end{align} which is the assertion.
\end{proof}

\end{revise}
The kernels we consider below for fast summation techniques satisfy this general relation~\eqref{eq:686} with $α=1$ and $α=1/2$. 
Note that, for $α=1/2$, the relations in \eqref{eq:701} and~\eqref{eq:704} reveal Hölder (and not Lipschitz) continuity of the embedding.
\paragraph{Prominent kernels.}
In applications, the choice of the kernel function is important, and it depends on the problem of interest.
The most prominent kernels used in MMDs are Gaussian~(Gauss), inverse multi-quadratic~(IMQ), Laplace~(Lap) and energy~(E) kernels.
The following table (Table~\ref{tab:0}) provides these kernels, together with the constants in the preceding Proposition~\ref{thm:666}.
\begin{table}[ht]
	\centering
	\begin{tabular}{llcc}
		\toprule
		Kernel & & Hölder exponent $α$ &Hölder constant $c$  \\
		\midrule
		Gauss & $k^{\text{Gauss}}({x},\tilde{x})\coloneqq {e}^{\nicefrac{-\|x-{\tilde{x}}\|^2}{ℓ^2}}$ & $1$ & $\nicefrac2{ℓ²}$ \\
		Laplace & $k^{\text{Lap}}({x},\tilde{x})\coloneqq {e}^{\nicefrac{-\|x-{\tilde{x}}\|}{ℓ}}$ & $1/2$ & $\nicefrac2{ℓ}$ \\
		Inverse multiquadratic & $k^{\text{IMQ}}(x,𝑥)≔ \nicefrac{1}{√{\|x-{\tilde{x}}\|^2+ \mathrm{c}^2}}$ & $1$ & $\nicefrac2{\mathrm{c}⁴}$ \\
		Energy & $k^{\text{e}}(x,𝑥)≔-\|x-{\tilde{x}}\|$ & $1/2$ & $√2$ \\
		\bottomrule
	\end{tabular}
	\caption{Prominent kernels and their constant in relation to the transportation problems, cf.~\eqref{eq:701}\label{tab:0}}
\end{table}

The discrepancy obtained by using the energy kernel, $k^{\text{e}}$, is known as \emph{energy} distance.
The energy distance was originally proposed by \citet{szekely2004testing}, which is occasionally known as Harald Cramér's distance.
In general, the energy kernel is not positive definite although, for compact spaces $𝓧$, it is positive definite with a minor modification, see \citet[Proposition~3.2]{neumayer2021optimal}.
Also, see \citet[Corollary~2.15]{graf2013efficient} for an explicit proof of conditional positivity of the energy kernel $k^{\text{e}}(⋅,⋅)$.
The equivalence of energy distance and Cramér's distance cannot be extended to higher dimensions because, unlike the energy distance, Cramér's distance is not rotation invariant.
\begin{revise}
\end{revise}
\subsection*{Convergence of 1-UOT to MMD energy distance \label{sec:convUOTMMD}}
The following results revisit the interconnection between OT and MMD for the energy kernel, elucidating the correspondence between UOT and MMD.
\begin{lemma}\label{lem:inter_ot_mmd}
 	Let $P$ and $Ꝓ$ be the  probability measures and $λ>0$ be the regularization parameter.
 	Then, the following relationships between the standard OT and the energy distance using debiased regularized OT $\mathrm{sd}_{1;λ}$
 	\begin{equation}\label{eq:lem_eq_1}
 		\lim_{λ → ∞} \mathrm{sd}_{1;λ}(P,Ꝓ) = {\ot}_{1}(P,Ꝓ)  \text{and}\quad
 		\lim_{λ → 0} \mathrm{sd}_{1;λ}(P,Ꝓ) = {\mmd_{k^{\text{e}}}(P,Ꝓ)}
 		\end{equation} hold true.
 \end{lemma}
 \begin{proof}
 	See \citet{ramdas2017wasserstein} and \citet[Proposition~2]{neumayer2021optimal}.
 \end{proof}
The following corollary discloses relationships between debiased regularized UOT and $\mmd_{k^\text{e}}(P,Ꝓ)$ for probability measures~$P$ and $Ꝓ$.
 \begin{corollary}\label{cor:477}
 	Let $λ>0$ be the entropy regularization parameter and $η>0$ be the marginal regularization parameter. Then, the relationships between OT with relaxed marginal constraints and energy distance using debiased regularized UOT
		\begin{align}
		\lim_{λ → ∞} \mathrm{sd}_{1;η;λ}(P,Ꝓ) = \uot_{1;η}(P,Ꝓ) \text{ and } 		 \lim_{\lambda \to 0} \lim_{\eta \to \infty}\mathrm{sd}_{1;\eta;\lambda}(P,\tilde{P}) = \mmd_{k^{\text{e}}}(P,\tilde{P})
	 \end{align}
	  hold true for probability measures~$P$ and~$\tilde{P}$.
\end{corollary}
\begin{revise}
\begin{proof}
	See Appendix~\ref{cor:477Proof}
\end{proof}
\end{revise} 

\rev{Up to this stage, we have established relationships and bounds between MMD and UOT through Propositions~\ref{prop:464},~\ref{prop:547},~\ref{prop:536},~\ref{thm:666},~\ref{thm:863}, and Corollary~\ref{cor:477}. 
These findings present value for practitioners, offering a reliable theoretical guideline for selecting readily accessible metrics.
These connections furnish theoretical assurances, aiding practitioners in choosing distances that are computationally feasible.
For instance, while computing MMD is generally straightforward, tasks involving Wasserstein distance and UOT can be daunting.
Therefore, by leveraging Hölder continuity results, practitioners can confidently opt for MMD, ensuring practicality across various applications.}

\section{Distances in discrete framework}\label{sec:discrete}
Computational aspects of the problems predominantly build on discrete measures, which motivates us to investigate faster and stable computational methods.
To this end, we reformulate the aforementioned problems in a discrete setting.
More precisely, the unbalanced optimal transport problem must be stated in its dual version to accommodate the fast summation technique presented below. On the other hand, the maximum mean discrepancy formulation does not require any additional modifications.

The discrete measures considered in the following sections are $μ=∑_{i=1}ⁿ μᵢ δ_{xᵢ}$ and $ν= ∑_{j=1}^ñ νⱼ δ_{𝑥ⱼ}$.

\subsection{Unbalanced optimal transport}
To reformulate problem~\eqref{eq:W4}, we introduce the distance or cost matrix $d ∈ ℝ^{n×ñ}$ with entries $d_{ij}≔ d(xᵢ, \tilde xⱼ)$.
With that, the discrete optimization problem~\eqref{eq:W4} is
\begin{align}
	\uot_{r;η;λ}(μ,ν)= \inf_π ∑_{ij} dʳ_{ij} π_{ij}
	&	+{1∕λ} ❨∑_{ij}π_{ij}\log {π_{ij} ∕ μᵢνⱼ}
		+ ∑_{i,j}μᵢ νⱼ - ∑_{ij}π_{ij}❩\\
	&	+ η₁ ❨∑ᵢ ∑ⱼ π_{ij}\log{∑_{j̕}π_{ij̕} ∕ μᵢ} + ∑ᵢ μᵢ - ∑ᵢ∑ⱼ π_{ij}❩\\
	&	+ η₂ ❨∑ⱼ ∑ᵢ π_{ij}\log{∑_{i̕}π_{i̕j} ∕ νⱼ} + ∑ᵢ νⱼ - ∑ⱼ∑ᵢ π_{ij}❩.
\end{align}
Here, the matrix $π∈ ℝ^{n×ñ}$ with positive entries represents the transportation plan, and the parameters are as defined in Definition~\ref{def:526}.

\medskip

 The duality reformulation gives rise to the accelerated algorithm of the UOT problem provided below. This explicit formulation justifies our advanced technique.

The following statement provides the dual formulation of the (primal) optimization problem~\eqref{def:526}, which is the foundation of our accelerated algorithm.
The proof builds on the Fenchel-Rockafellar Theorem, and we may refer to \citet[Theorem~3.2]{chizat2018scaling}.
\begin{figure}
	\begin{tikzpicture}[scale=0.45]
		\begin{axis}[%
		scatter/classes={%
			a={draw=blue},%
			b={mark=triangle*,draw=red}}]
		\addplot[scatter,only marks,%
			scatter src=explicit symbolic]%
		table[meta=label] {
		x y label
		-1 -1.3 a
		-2 -1.3 a
		-1.5 -0.3 a
		-0.5 -1.1 a
		-0.2 -1.5 a

		1.5 2.0 a
		2.0 2.3 a
		3.0 2.5 a
		3.5 2.9 a
		3.9 2.3 a
		2.2 2.5 a
		3.3 2.9 a
		3.3 3.2 a
		3.5 3.6 a
		3.9 3.5 a
		1.5 2.3 a

		2.0 3.9 a

		2.0 3.1 a

		4.5 6.0 b
		3.0 6.3 b

		6.0 6.5 b
		6.5 6.1 b
		6.9 6.3 b
		5.2 6.5 b
		6.3 6.7 b
		6.3 6.2 b
		6.5 6.4 b
		6.9 6.5 b
		4.5 5.1 b
		5.0 5.3 b
		5.0 6.1 b

		8 8.3 b
		7 8.3 b
		6.5 7.9 b
		6.5 8.3 b
		6.2 8.5 b
			};
		\end{axis}
		\end{tikzpicture}
		\begin{tikzpicture}[scale=0.45]
			\begin{axis}[%
			scatter/classes={%
				a={draw=blue},%
				b={mark=triangle*,draw=red}}]
			\addplot[scatter,only marks,%
				scatter src=explicit symbolic]%
			table[meta=label] {
			x y label
			-1 -1.3 a
			-2 -1.3 a
			-1.5 -0.3 a
			-0.5 -1.1 a
			-0.2 -1.5 a

			1.5 2.0 a
			2.0 2.3 a
			3.0 2.5 a
			3.5 2.9 a
			3.9 2.3 a
			2.2 2.5 a
			3.3 2.9 a
			3.3 3.2 a
			3.5 3.6 a
			3.9 3.5 a
			1.5 2.3 a

			2.0 3.9 a

			2.0 3.1 a

			4.5 6.0 b
			3.0 6.3 b

			6.0 6.5 b
			6.5 6.1 b
			6.9 6.3 b
			5.2 6.5 b
			6.3 6.7 b
			6.3 6.2 b
			6.5 6.4 b
			6.9 6.5 b
			4.5 5.1 b
			5.0 5.3 b
			5.0 6.1 b

			8 8.3 b
			7 8.3 b
			6.5 7.9 b
			6.5 8.3 b
			6.2 8.5 b
				};
				\draw[green] (axis cs:2.0,3.9) -- (axis cs:3.0,6.3);
				\draw[green] (axis cs:3.3,3.2) -- (axis cs:4.5,6.0 );
				\draw[green] (axis cs:3.5,2.9) -- (axis cs:5.0,6.1);
				\draw[green] (axis cs:3.9,3.5) -- (axis cs:5.0,5.3);
				\draw[green] (axis cs:3.5,3.6 ) -- (axis cs:4.5,5.1);
				\draw[green] (axis cs:3.3,2.9) -- (axis cs:5.2,6.5);
				\draw[green] (axis cs:3.9,2.3) -- (axis cs:6.0,6.5);
				\draw[green] (axis cs:3.0,2.5) -- (axis cs:6.3,6.7);
				\draw[green] (axis cs:2.0,3.1) -- (axis cs:6.5,6.1);
				\draw[green] (axis cs:2.2,2.5) -- (axis cs:6.3,6.2);
			\end{axis}
			\end{tikzpicture}
			\begin{tikzpicture}[scale=0.45]
				\begin{axis}[%
				scatter/classes={%
					a={draw=blue},%
					b={mark=triangle*,draw=red}}]
				\addplot[scatter,only marks,%
					scatter src=explicit symbolic]%
				table[meta=label] {
				x y label
				-1 -1.3 a
				-2 -1.3 a
				-1.5 -0.3 a
				-0.5 -1.1 a
				-0.2 -1.5 a

				1.5 2.0 a
				2.0 2.3 a
				3.0 2.5 a
				3.5 2.9 a
				3.9 2.3 a
				2.2 2.5 a
				3.3 2.9 a
				3.3 3.2 a
				3.5 3.6 a
				3.9 3.5 a
				1.5 2.3 a

				2.0 3.9 a

				2.0 3.1 a

				4.5 6.0 b
				3.0 6.3 b

				6.0 6.5 b
				6.5 6.1 b
				6.9 6.3 b
				5.2 6.5 b
				6.3 6.7 b
				6.3 6.2 b
				6.5 6.4 b
				6.9 6.5 b
				4.5 5.1 b
				5.0 5.3 b
				5.0 6.1 b

				8 8.3 b
				7 8.3 b
				6.5 7.9 b
				6.5 8.3 b
				6.2 8.5 b
					};
					\draw[green] (axis cs:2.0,3.9) -- (axis cs:3.0,6.3);
					\draw[green] (axis cs:3.3,3.2) -- (axis cs:4.5,6.0 );
					\draw[green] (axis cs:3.5,2.9) -- (axis cs:5.0,6.1);
					\draw[green] (axis cs:3.9,3.5) -- (axis cs:5.0,5.3);
					\draw[green] (axis cs:3.5,3.6 ) -- (axis cs:4.5,5.1);
					\draw[green] (axis cs:3.3,2.9) -- (axis cs:5.2,6.5);
					\draw[green] (axis cs:3.9,2.3) -- (axis cs:6.0,6.5);
					\draw[green] (axis cs:3.0,2.5) -- (axis cs:6.3,6.7);
					\draw[green] (axis cs:1.5,2.0) -- (axis cs:6.9,6.5);
					\draw[green] (axis cs:1.5,2.3) -- (axis cs:6.9,6.3 );
					\draw[green] (axis cs:2.0,2.3) -- (axis cs:6.5,6.4);
					\draw[green] (axis cs:2.0,3.1) -- (axis cs:6.5,6.1);
					\draw[green] (axis cs:2.2,2.5) -- (axis cs:6.3,6.2);
				\end{axis}
				\end{tikzpicture}
				\begin{tikzpicture}[scale=0.45]
					\begin{axis}[%
					scatter/classes={%
						a={draw=blue},%
						b={mark=triangle*,draw=red}}]
					\addplot[scatter,only marks,%
						scatter src=explicit symbolic]%
					table[meta=label] {
					x y label
					-1 -1.3 a
					-2 -1.3 a
					-1.5 -0.3 a
					-0.5 -1.1 a
					-0.2 -1.5 a

					1.5 2.0 a
					2.0 2.3 a
					3.0 2.5 a
					3.5 2.9 a
					3.9 2.3 a
					2.2 2.5 a
					3.3 2.9 a
					3.3 3.2 a
					3.5 3.6 a
					3.9 3.5 a
					1.5 2.3 a

					2.0 3.9 a

					2.0 3.1 a

					4.5 6.0 b
					3.0 6.3 b

					6.0 6.5 b
					6.5 6.1 b
					6.9 6.3 b
					5.2 6.5 b
					6.3 6.7 b
					6.3 6.2 b
					6.5 6.4 b
					6.9 6.5 b
					4.5 5.1 b
					5.0 5.3 b
					5.0 6.1 b

					8 8.3 b
					7 8.3 b
					6.5 7.9 b
					6.5 8.3 b
					6.2 8.5 b
						};
						\draw[green] (axis cs:6.5,7.9) -- (axis cs:-0.2,-1.5);
						\draw[green] (axis cs:8,8.3) -- (axis cs:-2,-1.3);
						\draw[green] (axis cs:-1.5,-0.3) -- (axis cs:6.2,8.5);
						\draw[green] (axis cs:-0.5,-1.1) -- (axis cs:6.5,8.3 );
						\draw[green] (axis cs:-1,-1.3) -- (axis cs:7,8.3);
						\draw[green] (axis cs:2.0,3.9) -- (axis cs:3.0,6.3);
						\draw[green] (axis cs:3.3,3.2) -- (axis cs:4.5,6.0 );
						\draw[green] (axis cs:3.5,2.9) -- (axis cs:5.0,6.1);
						\draw[green] (axis cs:3.9,3.5) -- (axis cs:5.0,5.3);
						\draw[green] (axis cs:3.5,3.6 ) -- (axis cs:4.5,5.1);
						\draw[green] (axis cs:3.3,2.9) -- (axis cs:5.2,6.5);
						\draw[green] (axis cs:3.9,2.3) -- (axis cs:6.0,6.5);
						\draw[green] (axis cs:3.0,2.5) -- (axis cs:6.3,6.7);
						\draw[green] (axis cs:1.5,2.0) -- (axis cs:6.9,6.5);
						\draw[green] (axis cs:1.5,2.3) -- (axis cs:6.9,6.3 );
						\draw[green] (axis cs:2.0,2.3) -- (axis cs:6.5,6.4);
						\draw[green] (axis cs:2.0,3.1) -- (axis cs:6.5,6.1);
						\draw[green] (axis cs:2.2,2.5) -- (axis cs:6.3,6.2);
					\end{axis}
					\end{tikzpicture}
					\tikz[x=1.1cm,y=1.2cm]{\draw[->](0,0)--(11.5,0);\node[text width=2cm] at (0,0.1) {$0$};\node[text width=2cm] at (13.0,0.1) {$\infty$};;\node[text width=2cm] at (6.5,0.2) {$\eta$};}
					\caption{Transportation (indicated by green \emph{lines}) of measures (circles and triangles) with equal mass. The optimal transportation plan bypasses the outliers for smaller parameters $η$, and appraises for larger parameters $η$, cf.\ Remark~\ref{rem:813}.\label{fig:UOTtoOT}}
	\end{figure}
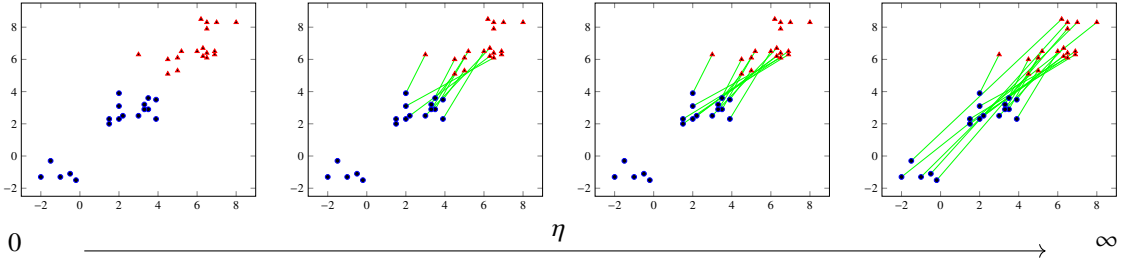
\begin{theorem}[Dualization of~\eqref{eq:W4}]\label{thm:dual_EUOT}
	For $\lambda> 0$, $\eta_1> 0$ and $\eta_2> 0$, the dual function of entropy-regularized UOT~\eqref{eq:W4} is
	\begin{align}\label{eq:dual_euot}
		\mathcal{D}(\beta,\gamma) \coloneqq \max_{\beta\in\mathbb R^n,\ \gamma\in\mathbb R^{\tilde n}} & - \eta_1\sum_{i=1}^n {e}^{-\nicefrac{\beta_i}{\eta_1}}  \mu_i-  \eta_2\sum_{j=1}^{\tilde n} {e}^{-\nicefrac{\gamma_j}{\eta_2}}  {\nu_j} +\, \eta_1\sum_{i=1}^n \mu_i+\,\eta_2\sum_{j=1}^{\tilde n} {\nu}_j \\ & -\frac1\lambda \sum_{i=1}^n\sum_{j=1}^{\tilde n} \mu_i ({e}^{\lambda\, \beta_i}\, {e}^{-\lambda\, d_{ij}^r}\, {e}^{\lambda\, \gamma_j}-1){\nu}_j,
	\end{align}
	which results from a strictly concave, smooth, and unconstrained optimization problem.
	Strong duality is obtained for $η₁>0$ and $η₂>0$.
\end{theorem}


The duality result~\eqref{eq:dual_euot} gives rise to the well-known Sinkhorn’s algorithm.
The unique maximizer~$(β⃰,γ⃰)$ is described by the first-order optimality condition, and by \emph{alternating optimization} (i.e., iterative procedure), we obtain a sequence $γ^{0}, β^{1},γ^{1},β^{2},‥$ as
\begin{equation}\label{eq:Sinkhorn_uot_iteration}
	β^{Δ+1}=-{η₁ ∕ 1+η₁λ}\log\big(k e^{λγ^Δ}\odot ν\big)
	 \text{ and } 
	γ^{Δ+1}=-{η₂ ∕ 1+η₂λ}\log\big(kᵀ e^{λ β^Δ}\odot μ\big),
\end{equation}
where $Δ∈ ℕ$ is the iteration count, $\odot$ refers to the element wise dot product and $k≔ e^{-λ dʳ}$ (entrywise), and we may set $γ^{(0)}≔ \zero_{ñ}$. This is the basis for Sinkhorn's iteration of regularized UOT.
Sinkhorn’s theorem ensures that this alternating optimization routine converges to the optimal solution (cf.\ \citet{Sinkhorn1967a}).


	Algorithm~\ref{alg:Sinkhorn_Log_sin_uot} summarizes the computational process,  provides the optimal dual variable ($β^{Δ⃰}$ and $γ^{Δ⃰}$), and desired quantities can be computed using it.
    For instance, the proximate solution of the regularized UOT Problem~\eqref{eq:W4} is \[π^{Δ⃰}={e^{λβ^{Δ⃰ }}} k {e^{λ γ^{Δ⃰}}} ⊙ (μ⊗ ν).\]

\begin{remark}[Hyperparameter tuning and outliers]\label{rem:814}
	In practice, we often fix $η₁ = η₂ ≕ η$ to reduce the necessity of hyperparameter tuning.

	From the consideration of Remark~\ref{rem:813}, we also deduce that, for finite values of $η$, problem~\eqref{eq:W4} tends to ignore larger entries of the distance matrix~$d$.
	That is, the solution of the problem does not involve outliers of the measures in the transportation.
	Figure~\ref{fig:UOTtoOT} illustrates the impact of the parameters $η$, highlighting the transportation of outliers.
	Furthermore,
	if one prefers to strengthen the mass conservation of $μ$ ($ν$, resp.), it is achievable by increasing the marginal parameter $η₁$ ($η₂$, resp.).
\end{remark}
\begin{algorithm}[tbh]
	\scriptsize
		\KwIn{Distance  $d_{ij}$ given in \eqref{eq:dist}, \rev{non-negative vectors} $μ∈ ℝⁿ$, $ν∈ℝ^ñ$, regularization parameter $η₁$, and $η₂$, entropy regularization parameter $\lambda>0$, $r≥ 1$, and starting value $γ=(γ₁,‥,γ_ñ)$}
		 Set
		 \begin{equation}
			\,k_{ij}= {e}^{-\lambda\,d_{ij}^r} i=1,‥,n, j=1,‥,ñ.
		 \end{equation}

		 \While{\emph{stopping criteria}}{
			\If{$\Delta$ is odd}{
				\begin{equation}\hspace*{-7.0cm}
					\begin{split}
						\hspace*{3.0cm}\beta_i^{\Delta}\leftarrow& -{\frac{\eta_1}{1+\eta_1\lambda}}\label{eq:995} \Big(\log\big(\sum_{j=1}^{\tilde n}k_{ij}\, {e}^{ \lambda \, \gamma_j^{\Delta-1}}\,\nu_j\big) \Big),\quad i=1,\dots, n; \\ {\gamma}_j^{\Delta}\leftarrow &{\gamma}_j^{\Delta-1},\quad j=1,\dots, \tilde n;\\[-3ex]
					\end{split}
				\end{equation}}
			\Else{
				\begin{equation}\hspace*{-7.0cm}
					\begin{split}
						\hspace*{3.0cm}\gamma_j^\Delta\leftarrow&{-\frac{η₂}{1+η₂λ}}\Big(\log\big(\sum_{i=1}^n k_{ij}\, \mathrm e^{ \lambda \, \beta_i^{\Delta-1}}\mu_i\big) \Big),\quad j=1,\dots, \tilde{n};\label{eq:996} \\ \beta_i^\Delta\leftarrow &\beta_i^{\Delta-1},\quad i=1,\dots, n;\\[-3ex]
					\end{split}
				\end{equation}}
				increment $\Delta \leftarrow \Delta +1 $}
				\caption{Sinkhorn algorithm UOT\label{alg:Sinkhorn_Log_sin_uot}}
		\KwOut{$\beta^{\Delta^*}$ and $\gamma^{\Delta^*}$}
	\end{algorithm}

\paragraph{Arithmetic complexity.} Even though the entropy regularization approach relaxes the computational burden, it still requires matrix-vector operation in~\eqref{eq:995} and~\eqref{eq:996} for each iteration, which are $𝓞(n²\, d)$ arithmetic operations.
The complete iteration routine of Algorithm~\ref{alg:Sinkhorn_Log_sin_uot} approximately requires \[𝓞\big(η\,λ(μ + ν)\,d\,n² \log n \,\big)\] arithmetic operations (cf.\ \citet{pham2020unbalanced}) to achieve machine precision, where $d$ is dimension.

The exposition on fast summation techniques in Section~\ref{sec:Fast_summation} below takes advantage of the specific structure of the matrix $k=e^{-λ dʳ}$ in~\eqref{eq:995} and~\eqref{eq:996}, which is known as \emph{Gibbs kernel} and/\,or kernel matrix.

\subsection{Computational aspects of MMDs}
For discrete measures $μ$ and $ν$ as above, the numerical computation of $\mmd²_k(μ,ν)$ in terms of the corresponding kernel function~$k$ is based on  
\begin{equation}\label{eq:1205}
	\mmd^2_k(μ,ν) = ∑_{i=1}ⁿ ∑_{j=1}ⁿ μ_i μⱼ k(xᵢ,xⱼ)
					- 2∑_{i=1}ⁿ ∑_{j=1}^ñ μᵢ νⱼ k(xᵢ,\tilde{x}ⱼ)
					+ ∑_{i=1}^ñ ∑_{j=1}^ñ νᵢ νⱼ k(\tilde{x}ᵢ,\tilde{x}ⱼ).
\end{equation}
The important observation here is that the individual terms in~\eqref{eq:1205} constitute matrix vector multiplications, where the matrix has the same shape as in~\eqref{eq:Sinkhorn_uot_iteration}, cf.\ also Algorithm~\ref{alg:Sinkhorn_Log_sin_uot}.
\paragraph{Arithmetic complexity.} In order to compute $\mmd²_k(μ,ν)$ or  $\mmd²_k(P,Ꝓ)$, the standard implementation~\eqref{eq:1205} requires $𝓞(n² d)$ arithmetic operations, where $d$ is the dimension and $n≈ ñ$.

\section{Fast summation method based on nonequispaced FFT}\label{sec:Fast_summation}
This section succinctly describes the nonequispaced Fourier transform (NFFT) based fast summation technique. In a nutshell, this technique allows reducing the computational burden of matrix-vector multiplications, which is $𝓞(n²)$ arithmetic operations.
The forthcoming discussion explicitly provides the arithmetic operations of our fast algorithms, which are proposed to solve entropy regularized UOT and MMDs.

The standard kernel matrix-vector operation resembles the multiplication of the Toeplitz matrix with the vector.
One can take advantage of this scenario to optimize $\mathcal{O}(n^2)$ to $\mathcal{O}( n\log n)$ arithmetic operations, using a fast algorithm (cf.\ \citet[Theorem~3.31]{plonka2018numerical}).
In other words, the fast algorithm embeds the kernel matrix into a circulant matrix, which is a special case of the Toeplitz matrix, and it diagonalizes the matrix  using the Fourier matrix. This approach leads to the  computation advancement, which is achieved using standard FFT (cf.\ \citet[pp.\ 141--142]{plonka2018numerical}).

Despite the benchmark computational advancement, standard FFT deteriorates by the restriction to \emph{equispaced} points ($x_i$, $\tilde x_j$).
Thus, the NFFT fast algorithm is introduced, which follows the similar approach for \emph{arbitrary} points, see \citet[Chapter~7]{plonka2018numerical}. We refer to \citet{nfft3} for a detailed interpretation of affiliated software.

\subsection{NFFT fast summation}
The fast summation method, which is based on NFFT, makes use of the \emph{radial} distance \[	d(x,𝑥):=\| x-𝑥\|, \]
which builds on the \emph{difference} of $x$ and $𝑥$;
the distance matrix $d \in \mathbb{R}^{n\times\tilde{n}}$ is
\begin{equation}\label{eq:dist}
	d(x_i,\tilde x_j)\coloneqq \| x_i- \tilde{x}_j\|, i=1,‥,n, j=1,‥,ñ.
\end{equation}
\begin{remark}[Radial cost functions]\label{rem:distance_matrix}
	The fast summation technique presented below crucially builds on radial cost functions.
	The cost functions contemplated here are increasing functions of the genuine distance~$d$, for example~$dʳ$.
	This type of cost functions is commonly referred to as \emph{radial} transport cost.
\end{remark}

The core concept of \emph{fast summation} is to efficiently approximate the radial kernels, i.e.,
\begin{equation}\label{eq:1492}
	\mathcal K^{\text{Gauss}}(y) \coloneqq  e^{-\nicefrac{\| y\|^2}{ℓ^2}}, \quad \mathcal K^{\text{IMQ}}(y) \coloneqq  \nicefrac{1}{\sqrt{\|y\|^2+ \mathrm{c}^2}}, \quad
\mathcal K^{\text{Lap}}(y) \coloneqq e^{-\nicefrac{\| y\|}{ℓ}},\quad \mathcal K^{\text{abs}}(y) \coloneqq  \| y\|, \end{equation}
where $y \coloneqq x - \tilde{x}$ ($ℓ>0$ and $c>0$).
\begin{figure}[!htb]
	\begin{center}
		\begin{tikzpicture}[scale=0.4]
			\begin{axis}[
				xlabel={$y$},
				xmin=-3, xmax=3,
				ymin=-0.1, ymax=1.1,
				legend pos=north east,
				]
				\addplot[color=blue, samples=500]{exp(-abs(x*x))};
				\legend{ ${e}^{-|y^2|}$
				}
				\end{axis}
				\end{tikzpicture}
				\begin{tikzpicture}[scale=0.4]
							\begin{axis}[
								xlabel={$y$},
								xmin=-2, xmax=2,
								ymin=0.4, ymax=1.1,
								legend pos=north east,
								]
								\legend{ $1/\sqrt{|y^2|+1}$								}
								\addplot[color=red, samples=500]{1/(sqrt{(abs(x*x)+1)})};

								\end{axis}
								\end{tikzpicture}
								\begin{tikzpicture}[scale=0.4]
									\begin{axis}[
										xlabel={$y$},
										xmin=-3, xmax=3,
										ymin=-0.1, ymax=1.1,
										legend pos=north east,
										]
										\legend{${e}^{-|y|}$
										}
										\addplot[color=green, samples=500]{exp(-abs(x))};

										\end{axis}
										\end{tikzpicture}
						\begin{tikzpicture}[scale=0.4]
									\begin{axis}[
										xlabel={$y$},
										xmin=-1, xmax=1,
										ymin=0., ymax=1.1,
										legend pos=north east,
										]
										\legend{ ${|y|}$}
										\addplot[color=violet, samples=500]{abs(x)};

										\end{axis}
										\end{tikzpicture}
			\caption{Gaussian~(blue), inverse multiquadric~(red), Laplace~(green) and energy or absolute value~(violet)}\label{Fig:Dist}
	\end{center}
\end{figure}
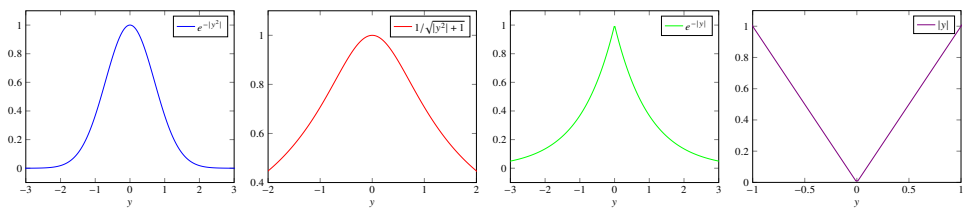
The fast summation approach, which performs matrix-vector operation, computes sums of the form
\begin{equation}\label{eq:fast_sum}
	s({x}_i) \coloneqq \sum_{j=1}^{{\tilde{n}}} \mathcal K(x_i-\tilde x_j)\,\tilde{\alpha}_j, \quad i=1,\dots,n,
\end{equation}
where coefficients $𝛼ⱼ∈ ℂ$ for all $j$.

\subsection{Approximation aspects of kernel approximation}

The central idea of NFFT is to efficiently approximate $𝓚(y)$ by an $h$-periodic trigonometric polynomial $𝓚_{RK}(y)$.
To illustrate, we approximate a univariate kernel using a trigonometric polynomial in one dimension via
\begin{equation}\label{eq:ansatz_NFFT}
		𝓚(y) ≈ 𝓚_{RK}(y) ≔ ∑_{\mathrm k \in \mathcal{I}_N} b_\mathrm k\, e^{2  \pi \mathrm i \mathrm k y / h}, \quad \mathcal{I}_N \coloneqq \Big\{ -\frac{N}{2},-\frac{N}{2}+1,\dots,-1,0,\dots,\frac{N}{2}-1\Big\},
	\end{equation}
	where $b_{\mathrm k}∈ℂ$ are appropriate Fourier coefficients, $y=x-\tilde{x}$ and bandwidth $N ∈ 2ℕ$.
	In order to efficiently compute the matrix-vector operations~\eqref{eq:fast_sum}, which is a bottleneck of UOT and MMDs, the fast summation based on NFFT approximates~$\mathcal K$ by the trigonometric polynomial $𝓚_{RK}$ with machine precision.

Note, that the desired kernels are non-periodic functions. Hence, an approximation using a trigonometric polynomial is not forthright.   
In order to obtain an $h$-periodic smooth kernel function $\tilde{\mathcal K}(y)$, we regularize ${\mathcal K}(y)$ as follows.
If the argument~$y$ satisfies $\|y\|\leq L$, we choose $h\ge 2L$.
For a good approximation, we embed $\mathcal{K}(y)$ into a smooth periodic function by
\[
	\tilde{\mathcal K}(y)\coloneqq
		\begin{cases}
			\mathcal K(y) &\text{if } y \in [-L,L], \\
			\mathcal K_\mathrm{B}(\|y\|) &\text{if } y ∈ (L,L+h),
		\end{cases}
\]
where $\mathcal K_\mathrm{B}$ is a univariate polynomial, which matches the derivatives of $\mathcal K(y)$ up to a certain degree. 
This makes its periodic continuation $\tilde{\mathcal K}(y)$ smooth. 
We refer to \citet[Chapter~7.5]{plonka2018numerical} for more mathematical details. 
In the one dimensional (univariate) case, we obtain the fast computation of distance matrix-vector operation by replacing the kernel $\mathcal{K}$ in~\eqref{eq:fast_sum} with corresponding Fourier representation~\eqref{eq:ansatz_NFFT}. More specifically,
\begin{align}\label{eq:fastsum}
	s(x_i) \approx	&\sum_{j=1}^{\tilde n} \tilde \alpha_j  \sum_{\mathrm k \in \mathcal{I}_N} b_\mathrm k
	 e^{2\pi \mathrm i \mathrm k (\bar{x}_i - \bar{\tilde{x}}_j)}\\
	=& \sum_{\mathrm k \in \mathcal{I}_N} b_\mathrm k \left(
		\sum_{j=1}^{\tilde n} \tilde \alpha_j e^{-2\pi \mathrm i \mathrm k \bar{\tilde{x}}_j}
		\right)e^{2\pi \mathrm{i} \mathrm k \bar{x}_i}, \quad i=1,\dots, n, \label{eq:1724}
\end{align}
with scaled nodes $\bar{x}_i \coloneqq h^{-1}x_i \in \mathbb{T}^d$, $\bar{\tilde{x}}_j \coloneqq h^{-1}\tilde{x}_j \in \mathbb{T}^d.$ Here, the torus $\mathbb{T}$ is
\begin{equation}\label{eq:boundaries}\mathbb{T} = \mathbb{R}\hspace{-0.3cm}\mod\mathbb{Z} \simeq \Big[-\frac{1}{2},\frac{1}{2}\Big).\end{equation}
In the aforementioned setup, the inner sum in~\eqref{eq:1724} is computed by the NFFT in $\mathcal O(d\,N\log N +\tilde n)$ arithmetical operations and the outer sum with $\mathcal O(d\,N\log N + n)$.
The aforementioned ansatz generalizes to the multivariate setting.
\paragraph{Curse of dimensionality}
Despite the robustness of NFFT, the evaluation is expensive for larger dimensions ($𝓧=ℝᵈ$, where $d≥4$). 
More specifically, for larger dimension, the grid belongs to the index set
\[
{\mathcal I}_{\mathcal N}:=\{-{\mathcal N}_{1}/_{2},\ldots,{\mathcal N}_{1}/_{2}-1\}\times\ldots\times\{-{\mathcal N}_{3}/_{2},\ldots,{\mathcal N}_{3}/_{2}-1\},
\] where $N=(N_1,N_2,N_3)\in2\mathbb{N}^3$, grows exponentially, which makes the evaluation expensive. 
To mitigate this burden, several research works have been presented in literature, which are predominantly based on \emph{mutual information score} and \emph{Analysis of Variance} (ANOVA) (cf.\ \citet{nestler2021learning}, \citet{potts2021approximation}).

\rev{In our context, we narrow our focus up to three dimensions, recognizing it as the most prevalent scenario across a wide array of applications. 
For more detailed specifications of approximations in higher dimensions (two and three), we refer \citet[Page~98]{nestlerdiss}.
}

\begin{remark}[Arithmetic complexity]\label{rem:arithmetic_complexity_NFFT}
	Evaluating~\eqref{eq:fast_sum} with Gaussian kernel~$\mathcal K^{\text{Gauss}}$ requires  $𝓞 (n\,d)$ arithmetic operations;
	with Laplacian $𝓚^{\text{Lap}}$, inverse multiquadratic $\mathcal K^{\text{IMQ}}$  and absolute value kernel $𝓚^{\text{e}}$, $𝓞 (d\,n\log n)$ arithmetic operations are required. 
	We refer to \citet[Chapter~7.5]{plonka2018numerical} for a comprehensive discussion.
\end{remark}
\begin{remark}[Heuristic stabilization]
The kernel approximation realm specifically tailored for certain boundaries, see \eqref{eq:boundaries}.
This specialization might introduce constraints aligned with our problem of interest.
In instances where the kernel function requires intensified focus within a narrower or wider interval, surpassing a threshold in boundary conditions, a prudent approach entails the integration of a scaling factor $h$, as previously elucidated, which is beneficial.
This augmentation serves to truncate the summation process~\eqref{eq:fast_sum}, thereby aligning the solution characteristics with the desired outcomes.
\end{remark}
\begin{remark}[NFFT parameters]
	The accuracy of the approximation in~\eqref{eq:fastsum} relies on various parameters, including the grid size or bandwidth $N$.
	We set those parameters in accordance with the specifications outlined in \citet{nfft3} to achieve machine precision.
	A common selection for the grid size is $N=2^8$.%
	\footnote{The internal window cutoff parameter $m$ plays a crucial role in approximation accuracy, which is independent of $n$ ($ ̃n$, resp.), and the typical choice $m=8$ yields machine precision.}
\end{remark}

\begin{algorithm}[!tbh]
	\scriptsize
		\KwIn{\rev{Support nodes $x\in\mathbb R^{n×d}  $ and $\tilde x \in\mathbb R^{\tilde n×d}$, non-negative vectors} $\mu\in\mathbb R^n$, $ \nu\in\mathbb R^{\tilde n}$, entropy regularization parameter $\lambda>0$, $r\ge 1$, and  starting value $\gamma=(\gamma_1,\dots,\gamma_{\tilde n})$}
		\While{\emph{stopping criteria}}
		{

			\If{$\Delta$ is odd}{compute
			\[	t_i^{\Delta-1}\leftarrow \sum_{j=1}^{\tilde n}  e^{-\lambda \|x_i-\tilde x_j\|^r}\cdot \nu_j e^{λ\,\gamma_j^{\Delta-1}},\quad i=1,\dots, n,\] by employing the fast summation~\eqref{eq:fastsum} and set
			\begin{equation}\hspace*{-5.5cm}\begin{split} \beta_i^{\Delta}\leftarrow&{-\frac{\eta_1}{1+\eta_1\lambda}} \Big( \log t_i^{\Delta-1} \Big) ,\quad i=1,\dots, n; \\ \gamma_j^{\Delta}\leftarrow &\gamma_j^{\Delta-1},\label{Eq:even_F_sink_log}\hspace*{-0.1cm}\quad j=1,\dots, \tilde{n}.\\[-2ex]\end{split}\end{equation}
			}
			\Else{compute
				\[\tilde t_j^{\Delta-1}\leftarrow \sum_{i=1}^n  {e}^{-\lambda \|x_i-\tilde x_j\|^r}\cdot\mu_i {e}^{ \lambda \, \beta_i^{\Delta-1}},\quad j=1,\dots,\tilde n,\] by employing the fast summation  \eqref{eq:fastsum} and set
				\begin{equation}\hspace*{-5.5cm}\begin{split} \gamma_j^{\Delta}\leftarrow&{-\frac{\eta_2}{1+\eta_2\lambda}}\Big( \log\tilde t_j^{\Delta-1} \Big), \quad j=1,\dots, \tilde{n}; \\ \beta_i^{\Delta}\leftarrow &\beta_i^{\Delta-1},\label{Eq:odd_F_sink_log}\hspace*{-0.1cm}\quad i=1,\dots, n.\\[-2ex]\end{split}\end{equation}}
				increment $\Delta \leftarrow \Delta +1$
			}
		\KwResult{$\beta^{\Delta^*}$ and $\gamma^{\Delta^*}$}
		\caption{NFFT-accelerated, Sinkhorn algorithm UOT\label{alg:logSinkhornFFT_UOT}}

	\end{algorithm}
\paragraph{Arithmetic complexity of Algorithm~\ref{alg:logSinkhornFFT_UOT}.} For the sake of simplicity, we assume $n=\tilde{n}$. If $r=1$ (Laplacian kernel), Algorithm~\ref{alg:logSinkhornFFT_UOT} requires \ \[\mathcal{O}\big(\eta\lambda\,(\mu + {\nu})\,d\, n \log² n \,\big)\] arithmetic operations to achieve machine accuracy.
If $r=2$ (Gaussian Kernel), it requires \[\mathcal{O}\big(\eta\lambda\,(\mu + {\nu})\,d\,n\,\log n  \,\big)\] arithmetic operations (cf.\ Remark~\ref{rem:arithmetic_complexity_NFFT}).

\paragraph{Arithmetic complexity of NFFT-accelerated MMD.}
To compute MMD using Gaussian $ \mathcal{K}^{\text{Gauss}}$, it requires \[\mathcal{O}(n\,d)\] arithmetic operations, and to compute MMD using Laplace $\mathcal{K}^{\text{Lap}}$, inverse multiquadric $\mathcal{K}^{\text{IMQ}}$ and absolute value $\mathcal{K}^{\text{abs}}$ kernels, it requires \[\mathcal{O}(n\,\log n\, d)\] arithmetic operations. 
\paragraph{Arithmetic complexity of upper bound of standard UOT and regularized UOT.}
The bottleneck in the Sections~\ref{sec:upperbound} and~\ref{sec:EntUOTConti} of the proposed formulation entails the inner product $⟨μ ⊗ ν | d^r⟩$, demanding $𝓞(n²d)$ arithmetic operations.
For $r=1$, however, it holds that $d(x,y)=-k_\text{e}(x,y)$, where $k_\text{e}$ is the energy kernel.
We thus can compute the inner product in \[𝓞(d\, n ㏒n)\] arithmetic operations using the fast absolute value $𝓚^{\text{abs}}$ kernel, cf.~\eqref{eq:1492}.
\section{Numerical exposition of the NFFT accelerated UOT and MMD}\label{Sec:Numerical_exposition}
In this section, we demonstrate numerical results for the approaches and methods considered in the preceding sections.
The corresponding implementations are available online.\footnote{Cf.\ \href{https://github.com/rajmadan96/FastUOTMMD.git}{https://github.com/rajmadan96/FastUOTMMD.git}} The implementation of NFFT techniques are based on the freely available repository ‘NFFT3.jl’.\footnote{Cf.\ \href{https://github.com/NFFT/NFFT3.jl}{https://github.com/NFFT/NFFT3.jl}}
All experiments were performed on a standard desktop computer with  Intel(R) Core(TM) i7-7700 CPU and 15.0 GB of RAM. Additionally, we emphasize that our proposed methods do not require expensive hardware or having access to supercomputers which is highly advantageous in terms of access to wide and low-threshold applications. 


\subsection{UOT acceleration}
In this section, we demonstrate the enhanced performance and accuracy of regularized UOT using synthetic and benchmark data.

\subsubsection{Empirical measures\label{sec:synthetic_uot}}
We demonstrate the efficiency~(in terms of time and memory allocations) of our proposed Algorithm~\ref{alg:logSinkhornFFT_UOT} using synthetic data.

Consider the unbalanced measures
\[	μ=\sum_{i=1}^n \mu_i\,\delta_{(U^1_i,U^2_i)}  \text{ and } 
	ν=\sum_{j=1}^{\tilde n} \nu_j\,\delta_{(\tilde U_j^1,\tilde U_j^2)}\]
on $\mathbb R\times\mathbb R$, where $U^1_i$, $U^2_i$, $\tilde U_j^1$, $\tilde U_j^2$, as well as the weights $μᵢ$ and $νⱼ$, are independent samples from the uniform distribution on $[0,1]$ ($i=1,‥n$, $j=1,\dots,\tilde n$).
\begin{table}[htb]
	\footnotesize
	\begin{centering}
		\begin{tabular}{lcccccccc}
			\toprule
			$n=\tilde n$: &    \multicolumn{2}{c}{\num{1000}} & \multicolumn{2}{c}{\num{10000}} & \multicolumn{2}{c}{\num{100000}} & \multicolumn{2}{c}{\num{1000000}}\tabularnewline
			\midrule & Time & Memory & Time & Memory & Time & Memory & Time & Memory \tabularnewline & (Sec) & (MB)& (Sec) & (MB)& (Sec) & (MB)& (Sec) & (MB)\tabularnewline
			\midrule\textbf{One dimension}~($d=1$)\tabularnewline
			\midrule
			UOT Sinkhorn's Alg.  ~\ref{alg:Sinkhorn_Log_sin_uot}
			& \num{0.20}& \num{210.73}& \num{23.08}& \num{20655.10} & \multicolumn{4}{c}{\emph{\textbf{out of memory}}} \tabularnewline
			UOT NFFT Sinkhorn's Alg.~\ref{alg:logSinkhornFFT_UOT}
					& \textbf{\num{0.09}} & \textbf{\num{12.64}} & \textbf{\num{0.43}} & \textbf{\num{150.28}} & \textbf{\num{4.47}} & \textbf{\num{1495.04}} & \textbf{\num{39.53}} & \textbf{\num{15022.08}} \tabularnewline[2\doublerulesep]
					\midrule\textbf{Two dimension}~($d=2$) \tabularnewline
						\midrule
									UOT Sinkhorn's Alg.  ~\ref{alg:Sinkhorn_Log_sin_uot}
					& \num{0.63}& \num{256.51}& \num{25.60}& \num{25233.40} & \multicolumn{4}{c}{\emph{\textbf{out of memory}}} \tabularnewline
												UOT NFFT Sinkhorn's Alg.~\ref{alg:logSinkhornFFT_UOT}
					& \textbf{\num{0.72}} & \textbf{\num{12.69}} & \textbf{\num{2.04}} & \textbf{\num{150.73}} & \textbf{\num{9.08}} & \textbf{\num{1507.32}} & \textbf{\num{77.24}} & \textbf{\num{15067.13}} \tabularnewline[2\doublerulesep]
					\midrule \textbf{Three dimension}~($d=3$)\tabularnewline
						\midrule
									UOT Sinkhorn's Alg.~\ref{alg:Sinkhorn_Log_sin_uot}
					& \textbf{\num{10.63}}& \num{465.98}& \textbf{\num{ 67.21}}& \num{37324.46} & \multicolumn{4}{c}{\emph{\textbf{out of memory}}} \tabularnewline
												UOT NFFT Sinkhorn's Alg.~\ref{alg:logSinkhornFFT_UOT}
					& \num{31.32} & \textbf{\num{18.32}} & \textbf{\num{100.54}} & \textbf{\num{191.13}}\ & \textbf{\num{202.17}} & \textbf{\num{1505.28}} & \textbf{\num{419.27}} & \textbf{\num{15114.24}} \tabularnewline[2\doublerulesep]
			\bottomrule
		\end{tabular}
	\par\end{centering}
\smallskip
	\caption{\label{tab:Sinkhorn3}Performance analysis of regularized UOT vs.\ NFFT accelerated--regularized UOT for varying dimensions; the accelerated algorithm manages problem sizes, which are out of reach for standard implementations. Parameters: $r=2$, $λ=20$ and $η= 1$; the best results are in bold}
\end{table}

Table~\ref{tab:Sinkhorn3} displays computation times and memory allocations of UOT Sinkhorn’s Algorithm~\ref{alg:Sinkhorn_Log_sin_uot} and UOT NFFT Sinkhorn’s Algorithm~\ref{alg:logSinkhornFFT_UOT}.
Our NFFT accelerated algorithm significantly outperforms the standard algorithm.
Notably, our proposed algorithm easily reaches problem sizes, which are completely inaccessible for standard implementations. 

\subsubsection{Accuracy analysis with benchmark datasets}
To unleash the exact ground truth of our proposed algorithm, we perform an accuracy analysis using benchmark datasets.

We employ the DOTmark dataset, which is explicitly designed to validate the performance and accuracy of new OT techniques and algorithms (cf.\ \citet{schrieber2016dotmark}).
The DOTmark dataset consists of 10 subsets of gray scale images, ranging from  $32 × 32$ to $512 × 512$ resolution.

\paragraph{Transformation of images to (unbalanced) vectors.}
The grayscale image is given as a matrix, with each entry representing the intensity of a pixel in the range $[0,1]$ (black: 0, white: 1).
The standard OT problem would --~inappropriately~-- normalize the matrix.
Our approach to UOT does not require normalizing the measure.
Moreover, background pixel intensities are the $\ell_1$ distance between pixels $i$ and $j$ of the respective grids $(32 \times 32, \dots$, and $512 \times 512)$.

We compute the \emph{residual},
\[\text{Residual} = \Bigg(\frac{\|\beta^*_{\text{NFFT}}-\beta^*\|}{\|\beta^*\|} + \frac{\|\gamma^*_{\text{NFFT}} -\gamma^* \|}{\|\gamma^*\|} \Bigg),\]
where $\beta^*$ and  $\gamma^*$ are optimal dual variable of UOT Sinkhorn's Algorithm~\ref{alg:Sinkhorn_Log_sin_uot} and $\beta^*_{\text{NFFT}} $ and  $\gamma^*_{\text{NFFT}}$  are optimal dual variable of UOT NFFT Sinkhorn's Algorithm~\ref{alg:logSinkhornFFT_UOT}, to substantiate the accuracy of our proposed algorithm.
\begin{table}[!htb]
	\begin{centering}
	\begin{tabular}{lccc}
		\toprule
		$n=\tilde n:$ &  \num{1024} & \num{4029} & \num{16384} \tabularnewline
		\midrule
		Dataset: DOTmark&  \multicolumn{3}{c}{average residual } \tabularnewline
		\midrule
		& \num{4.89e-14} & \num{3.86e-14} & \num{3.71e-14}
		\tabularnewline
		\bottomrule
	\end{tabular}
\par\end{centering}

\smallskip

\caption{Accuracy analysis: NFFT accelerated--regularized UOT; parameters $r=2$, $\lambda=20$ and $\eta=1$ \label{Tab:Accuracy_analysis2}}
\end{table}

Table~\ref{Tab:Accuracy_analysis2} comprises the results of our accuracy analysis. From the results, it is evident that our proposed algorithm provides machine accuracy as promised.
For the sake of brevity, we skip the precise results and instead present the average residual.
Although, we would like to emphasize that our proposed algorithm is capable of withstanding from smooth to rough nature of problems provided by DOTmark datasets.
\begin{remark}[Regularization parameter $λ$]
	In this experiment above we observe stable computation of NFFT-accelerated regularized UOT using Algorithm~\ref{alg:logSinkhornFFT_UOT} across the range of $λ ∈ (0, 200]$.
	As mentioned earlier, this threshold for $λ$ is adequate in achieving the approximate optimal solution within machine accuracy.
 \end{remark}
\begin{remark}[New initialization of Sinkhorn's algorithm]
	The UOT experiments employ the regularized UOT upper bound ($c^*$) as the initialization factor $γ^{(0)} = c^*⋅\one_{ñ}$ for Algorithm~\ref{alg:logSinkhornFFT_UOT}, deviating from the conventional choice of $γ^{(0)} = \zero_{ñ}$. This departure resulted in a discernible enhancement, indicating a slight improvement over the standard selection.
	However, a recent study by \citet{IntSinkhorn} has examined the initialization factor of Sinkhorn’s algorithm, presents  promising results.
	Notably, their approach involves data-dependent initializers, distinguishing it from our methodology.
\end{remark}

\subsubsection{$λ$-scaling}
A robust heuristic approach for overcoming the loss of spectral information and mitigating the slowdown in the convergence process is the utilization of $λ$-scaling (also known as deformed iteration) (cf.\ \citet[Section~3.2]{schmitzer2019stabilized}, \citet[Section~3]{sharify2011solution}).
This method entails the iterative resolution of the regularized problem with progressively increasing values for the regularization parameter $λ∈ Λ⊂ ℝ₊$. 

The following experiment showcases the adaptability of our proposed method with this $λ$‑scaling technique.

The experimental setup adheres to the simulation depicted in Figure~\ref{fig:6a}. 
For this investigation, we focus on probability measures—an instantiation of UOT also recognized as robust OT, see \citet{balaji2020robust}.
\begin{figure}[ht!]
	\centering
	\subfloat[][Marginal probability measures $P$ (blue) and $Ꝓ$ (red)]
	{	\includegraphics[width=0.30\textwidth,height=0.23\textwidth]{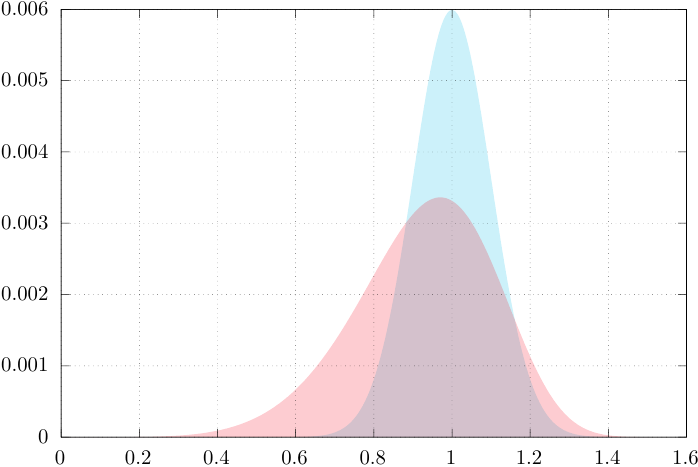}\label{fig:6a}}\hspace{2cm}
			\subfloat[][\textbf{\num{49}} total iterations with $λ$-scaling the regularization parameter $λ$ along $Λ$]
	{	\includegraphics[width=0.30\textwidth,height=0.23\textwidth]{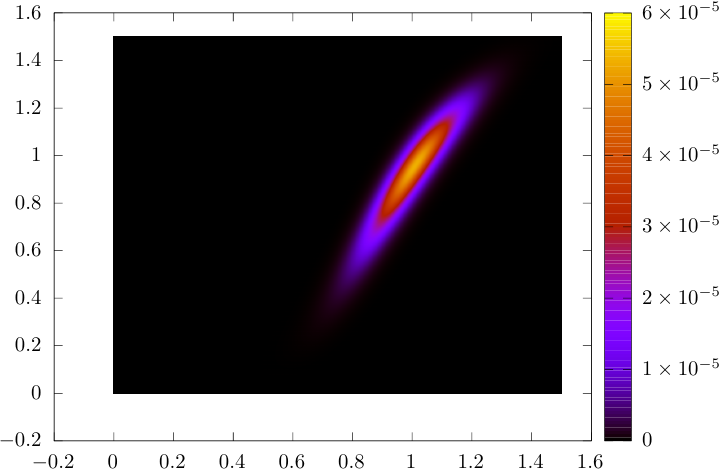}\label{fig:6b}}

	\subfloat[][\textbf{\num{26}} total iterations to obtain the (suboptimal) solution for $λ=1$]
	{	\includegraphics[width=0.30\textwidth,height=0.23\textwidth]{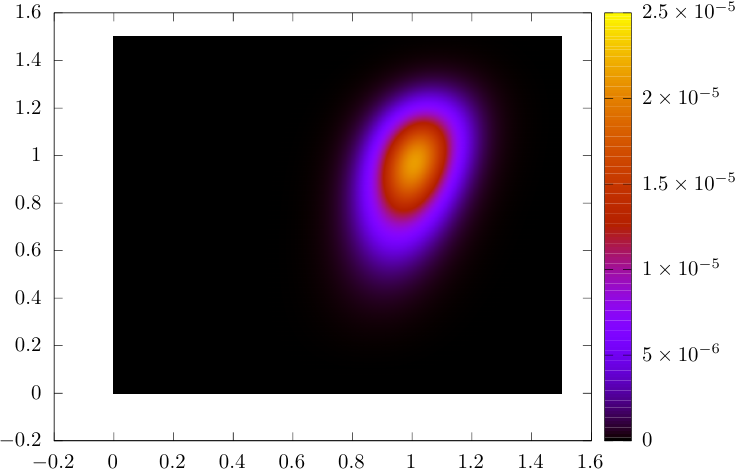}\label{fig:6c}}
	\hfill
	\subfloat[][\textbf{\num{44}} total iteration for $λ=20$]
	{	\includegraphics[width=0.30\textwidth,height=0.23\textwidth]{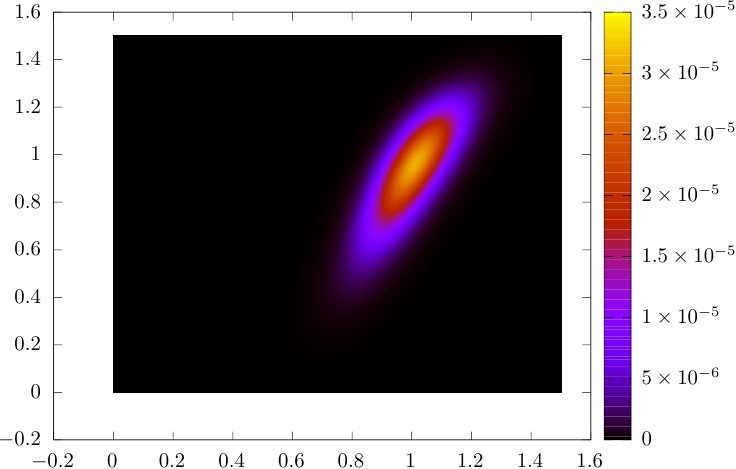}\label{fig:6d}}
	\hfill
	\subfloat[][\textbf{\num{78}} iterations with $λ=100$ to achieve the same accuracy as in Figure \ref{fig:6b}.]
	{	\includegraphics[width=0.30\textwidth,height=0.23\textwidth]{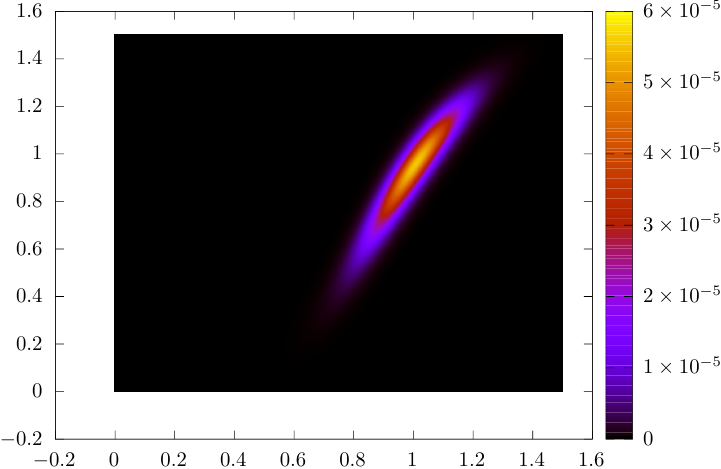}\label{fig:6e}}
	\caption{Heat maps of optimal matrices $π^*$ NFFT-accelerated UOT Sinkhorn's $λ$-scaling technique vs. NFFT-accelerated UOT Sinkhorn's Algorithm for varying regularization parameter $λ$; parameters  $λ ∈ Λ= \{1,20,100\}$ and $η_1=η_2=25$.}\label{fig:scaling}
\end{figure}

\medskip

Figure~\ref{fig:scaling} presents the heat map of optimal matrices $π^*$ along with total number of iteration for each experiments.
It demonstrates that the NFFT-accelerated UOT Sinkhorn's Algorithm~\ref{alg:logSinkhornFFT_UOT} utilizing $λ$-scaling technique converges faster than the NFFT-accelerated UOT Sinkhorn's Algorithm~\ref{alg:logSinkhornFFT_UOT}. The adaptability of this technique further enhances the performance of our proposals.
\subsection{MMDs accelerations}
This section demonstrates the performance and accuracy of NFFT accelerated MMDs. We perform the experiments using the measures $\mu$ and $\nu$ as described in Section~\ref{sec:synthetic_uot} above.
\begin{table}[!ht]
	\begin{centering}
		\begin{tabular}{lcccc}
			\toprule
			\textbf{Kernels} &    Gaussian & Laplace & inverse multi-quadratic & energy kernel\tabularnewline
			\midrule  & \multicolumn{4}{c}{average residuals} \tabularnewline
			\midrule
					{probability vectors} & \num{2.42e-11} & \num{1.10e-6} & \num{6.03e-7} & \num{5.30e-7}  \tabularnewline[2\doublerulesep]\midrule
			{arbitrary~(unbalanced) vectors}
								& \num{1.31e-8} & \num{2.04e-6} & \num{2.84e-7} & (not a distance)  \tabularnewline[2\doublerulesep]
								\bottomrule
		\end{tabular}
	\par\end{centering}
\smallskip
	\caption{Accuracy analysis: NFFT accelerated MMDs; hyperparameters $\mathrm{c}=1$, $\ell=\frac{1}{2}$; dimension~$d=2$; problem size ranging from $n = \tilde{n} =\num{100}$ to $n=\tilde{n} =\num{10000}$ \label{Tab:Accuracy_analysis3}}
\end{table}
For the accuracy analysis, we conduct experiments using probability vectors and unbalanced vectors.
To validate the quality of our NFFT MMDs, we compute the errors
\[\text{Residual} = \frac{|\mmd^2_{{k}}(\mu,{\nu})-\mmd^2_{\mathcal{K}}(\mu,{\nu})|}{\mmd^2_{{k}}(\mu,{\nu})}.\]
From Table~\ref{Tab:Accuracy_analysis3}, we infer that the quality of the approximation is stable, and the errors are negligible.
\begin{table}[!ht]
	\footnotesize
	\begin{centering}
		\begin{tabular}{lcccccccc}
			\toprule
			$n=\tilde n$: &    \multicolumn{2}{c}{\num{1000}} & \multicolumn{2}{c}{\num{10000}} & \multicolumn{2}{c}{\num{100000}} & \multicolumn{2}{c}{\num{1000000}}\tabularnewline
			\midrule & time & memory & time & memory & time & memory & time & memory \tabularnewline & (sec) & (MB)& (sec) & (MB)& (sec) & (MB)& (sec) & (MB)\tabularnewline
			\midrule Gaussian kernel: $\mmd^2_{k^{\text{Gauss}}}(\mu,{\nu})$\tabularnewline
			\midrule
			Std. computation 
			& \num{0.33}& \num{663.78}& \num{58.65}& \num{66375.68} & \multicolumn{4}{c}{\emph{\textbf{out of memory}}}\tabularnewline
			NFFT Acc. computation 
					& \textbf{\num{0.06}} & \textbf{\num{0.23}} & \textbf{\num{0.02}} & \textbf{\num{2.29}} & \textbf{\num{0.78}} & \textbf{\num{22.89}} & \textbf{\num{1.29}} & \textbf{\num{251.77}} \tabularnewline\midrule
					Std. computation
					& \num{1.10}& \num{806.88}& \num{73.40}& \num{80107.52} & \multicolumn{4}{c}{\emph{\textbf{out of memory}}} \tabularnewline
					NFFT Acc. computation
							& \textbf{\num{0.17}} & \textbf{\num{1.04}} & \textbf{\num{0.15}} & \textbf{\num{3.20}} & \textbf{\num{1.55}} & \textbf{\num{32.0}} & \textbf{\num{3.39}} & \textbf{\num{320.43}} \tabularnewline\midrule
							Std. computation
							& \textbf{\num{3.05}}& \num{801.11}& {\num{65.87}}& \num{80188.41} & \multicolumn{4}{c}{\emph{\textbf{out of memory}}} \tabularnewline
							NFFT Acc. computation
									& \num{14.04} & \textbf{\num{0.41}} & \textbf{\num{16.73}} & \textbf{\num{135.68}} & \textbf{\num{17.20}} & \textbf{\num{173.20}} & \textbf{\num{34.96}} & \textbf{\num{543.99}} \tabularnewline[2\doublerulesep]
									\midrule Laplace kernel: $\mmd^2_{k^{\text{Lap}}}(\mu,{\nu})$  \tabularnewline
						\midrule
			Std. computation
			& \num{0.36}& \num{640.89}& \num{58.15}& \num{64081.92} & \multicolumn{4}{c}{\emph{\textbf{out of memory}}} \tabularnewline
			NFFT Acc. computation
					& \textbf{\num{0.17}} & \textbf{\num{0.24}} & \textbf{\num{0.06}} & \textbf{\num{2.29}} & \textbf{\num{0.27}} & \textbf{\num{22.91}} & \textbf{\num{1.72}} & \textbf{\num{228.88}} \tabularnewline
						\midrule
						Std. computation
						& \num{1.56}& \num{779.59}& \num{72.00}& \num{77819.90} & \multicolumn{4}{c}{\emph{\textbf{out of memory}}} \tabularnewline
						NFFT Acc. computation
								& \textbf{\num{0.19}} & \textbf{\num{1.06}} & \textbf{\num{0.22}} & \textbf{\num{3.20}} & \textbf{\num{1.22}} & \textbf{\num{32.04}} & \textbf{\num{3.75}} & \textbf{\num{320.43}} \tabularnewline
								\midrule
								Std. computation
						& \textbf{\num{5.53}}& \num{778.22}& {\num{67.15}}& \num{77833.216} & \multicolumn{4}{c}{\emph{\textbf{out of memory}}} \tabularnewline
						NFFT Acc. computation
								& 	{\num{13.30}} & \textbf{\num{0.41}} & \textbf{\num{17.22}} & \textbf{\num{4.86}} & \textbf{\num{18.78}} & \textbf{\num{41.94}} & \textbf{\num{33.16}} & \textbf{\num{412.73}} \tabularnewline[2\doublerulesep]
								\midrule IMQ kernel: $\mmd^2_{k^{\text{IMQ}}}(\mu,{\nu})$
					\tabularnewline\midrule Std. computation
					& \num{0.43}& \num{595.11}& \num{65.34}& \num{59509.76} & \multicolumn{4}{c}{\emph{\textbf{out of memory}}} \tabularnewline
					NFFT Acc. computation
							& \textbf{\num{0.06}} & \textbf{\num{0.25}} & \textbf{\num{0.67}} & \textbf{\num{2.52}} & \textbf{\num{0.49}} & \textbf{\num{25.18}} & \textbf{\num{2.23}} & \textbf{\num{228.81}} \tabularnewline
							\midrule
							Std. computation
						& \num{0.90}& \num{745.58}& \num{75.66}& \num{73242.62} & \multicolumn{4}{c}{\emph{\textbf{out of memory}}} \tabularnewline
						NFFT Acc. computation
								& \textbf{\num{0.24}} & \textbf{\num{1.08}} & \textbf{\num{0.25}} & \textbf{\num{3.43}} & \textbf{\num{1.03}} & \textbf{\num{34.33}} & \textbf{\num{4.07}} & \textbf{\num{343.32}} \tabularnewline
								\midrule

						Std. computation
						& \textbf{\num{3.87}}& {\num{732.44}}& \textbf{\num{70.90}}& \num{73267.2} & \multicolumn{4}{c}{\emph{\textbf{out of memory}}} \tabularnewline
						NFFT Acc. computation
								& {\num{13.76}} & \textbf{\num{0.43}} & \textbf{\num{20.63}} & \textbf{\num{135.66}} & \textbf{\num{28.10}} & \textbf{\num{175.25}} & \textbf{\num{32.63}} & \textbf{\num{566.64}} \tabularnewline[2\doublerulesep]
								\midrule Energy kernel: $\mmd^2_{k^{\text{e}}}(P,\tilde{P})$
					\tabularnewline
					\midrule
					Std. computation
			& \num{0.47}& \num{595.13}& \num{54.78}& \num{51502.36} & \multicolumn{4}{c}{\emph{\textbf{out of memory}}} \tabularnewline
			NFFT Acc. computation
					& \textbf{\num{0.06}} & \textbf{\num{0.27}} & \textbf{\num{0.08}} & \textbf{\num{2.60}} & \textbf{\num{0.53}} & \textbf{\num{25.12}} & \textbf{\num{1.25}} & \textbf{\num{251.71}} \tabularnewline
								\midrule
								Std. computation
						& {\num{2.65}}& \num{762.13}& {\num{61.69}}& \num{73274.36} & \multicolumn{4}{c}{\emph{\textbf{out of memory}}} \tabularnewline
						NFFT Acc. computation
								& \textbf{\num{0.12}} & \textbf{\num{2.28}} & \textbf{\num{0.20}} & \textbf{\num{3.51}} & \textbf{\num{1.09}} & \textbf{\num{35.10}} & \textbf{\num{3.44}} & \textbf{\num{343.21}} \tabularnewline\midrule
								Std. computation
						& \textbf{\num{0.58}}& \num{734.39}& \num{62.43}& \num{73236.48} & \multicolumn{4}{c}{\emph{\textbf{out of memory}}} \tabularnewline
						NFFT Acc. computation
								& \num{13.75} & \textbf{\num{0.46}} & \textbf{\num{15.21}} & \textbf{\num{36.10}} & \textbf{\num{17.39}} & \textbf{\num{45.49}} & \textbf{\num{36.18}} & \textbf{\num{436.88}} \tabularnewline[2\doublerulesep]
			\bottomrule
		\end{tabular}
	\par\end{centering}
\smallskip
	\caption{\label{tab:Sinkhorn4}Performance analysis (computation time and memory allocation) of MMD for varying kernels: standard computation vs.\ NFFT-accelerated of MMDs; dimension $d=1,2,3$ in subsequent order for each kernel;  the hyperparameters are $\mathrm{c}=1$, $ℓ=\frac{1}{2}$; the best results are in bold; NFFT accelerated computation outperforms the standard computation, although in dimension $d=3$, it is less significant for problem size $n=\tilde{n}=1000$ }
\end{table}

We then illustrate the performance of NFFT accelerated MMDs.  Table~\ref{tab:Sinkhorn4} (below) comprehends the performance of the standard MMDs and NFFT accelerated MMDs using different problem size and dimensions~$(d=1,2,3)$.
It evidently confirms that the performances of our NFFT accelerated MMDs are significantly better than the standard computations.
Nevertheless, we notice that NFFT accelerated computations are slightly expensive (in terms of time allocation) for problem size $n=\tilde{n}=1000$ in dimension $d=3$. However, this slightly expensive time allocation is negligible when compared to memory allocations of standard computations.
\begin{revise}
\subsubsection*{Convergence of 1-UOT to energy distance}

Now, we empirically validate Corollary~\ref{cor:477} by utilizing synthetic data.
For this experiment, we fix $λ=0.001$, and consider empirical probability measures $P$ and $\tilde{P}$ on the closed interval $[0, 1]$ for uniformly samples.

\begin{figure}[!htb]
	\centering
	\includegraphics[width=0.6\textwidth]{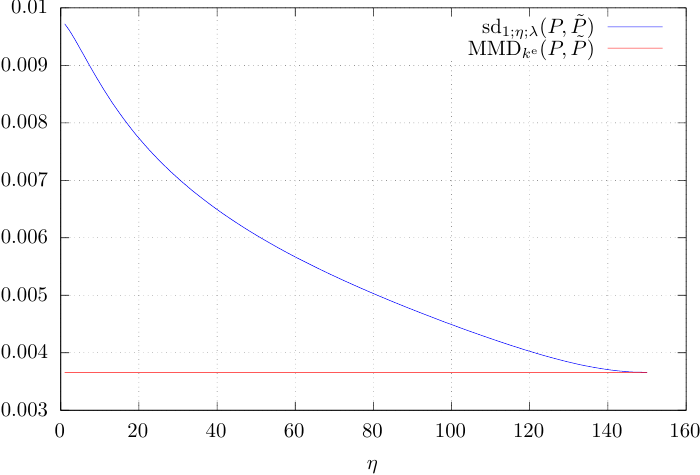}
	\caption{\rev{$\mathrm{sd}_{1;\eta;\lambda}(P,\tilde{P})$ approaches $\mmd_{k^{\text{e}}}(P,\tilde{P})$ for increasing penalization parameters~$\eta$, cf~\eqref{eq:586} ($λ=0.001$)}\label{fig:uot_MMD}}
\end{figure}

Figure~\ref{fig:uot_MMD} illustrates $sd_{1;\eta;\lambda}(P,\tilde{P})$ and $\mmd_{k^{\text{e}}}(P,\tilde{P})$ for various values of the marginal regularization parameter. This visualization corroborates Corolary~\ref{cor:477}.

\subsection{Disparity between MMD and transportation problems}

Propositions~\ref{thm:666}~and~\ref{thm:863} establish  Hölder continuity between MMD and transportation problem (Wasserstein distance and UOT).
This section investigates into scrutinizing the disparities and inequalities of the proposals across different parameter settings.

\paragraph{Probability measures, cf.\ Propositions~\ref{thm:666}.}
We consider empirical probability measures $P$ and $Q$ on the closed interval $[0, 1]$ following a uniform distribution.
The observed average disparities from 1000 simulations in Table~\ref{Tab:DisMMDWas} confirm inequality~\eqref{eq:701} in Proposition~\ref{thm:666}.
 \begin{table}[!htb]
	\begin{centering}
		\begin{revise}
	\begin{tabular}{lccc}
		\toprule
		$n=\tilde n$: &  \num{100} & \num{500} & \num{1000} \tabularnewline
		\midrule
		$\rhs-\lhs $& \num{0.2248} & \num{0.1906} & \num{0.2481}
		\tabularnewline
		\bottomrule
	\end{tabular}
	\end{revise}
\par\end{centering}

\smallskip

\caption{\rev{Average disparity (\num{1000} simulations) between $\lhs\coloneqq \mmd_{k^{\text{Gauss}}}(P,Q)$ and $\rhs≔c w_{2α}(P,Q)^{α}$ for the Gaussian kernel, which has parameters $α=1$, $c=2$ and $ℓ=1$ }\label{Tab:DisMMDWas}}
\end{table}
\paragraph{Unbalanced measures, cf.\ Propositions~\ref{thm:863}.}
For synthetic data as described in Section~\ref{sec:synthetic_uot},
Table~\ref{Tab:DisMMDUOT} displays the average disparities observed from 1000 simulations with fixed parameters.
The results substantiate the observations in inequality~\eqref{eq:704}.
\begin{table}[!htb]
	\begin{centering}
		\begin{revise}
	\begin{tabular}{lccc}
		\toprule
		$n=\tilde n$: &  \num{100} & \num{500} & \num{1000} \tabularnewline
		\midrule
		$\rhs-\lhs $& \num{0.2867} & \num{0.2231} & \num{0.1929}
		\tabularnewline
		\bottomrule
	\end{tabular}
	\end{revise}
\par\end{centering}
\caption{\rev{Average disparity (\num{1000} simulations) between $\lhs\coloneqq\mmd_{k^{\text{Gauss}}}(μ,ν)$ and $\rhs\coloneqq c⋅ \sqrt{u^*}⋅\uot_{2α;\nicefrac{η}{c²} }(μ,ν)^{α} \mmd_{k^{\text{Gauss}}}(μ,μ^*)+ \mmd_{k^{\text{Gauss}}}(ν,ν^*)$ (parameters $α=1$, $c=2$  $ℓ=1$, and $η= 8$, cf.\ Table~\ref{tab:0})}\label{Tab:DisMMDUOT}}
\end{table}

\begin{remark}
	The disparity analysis aims to uncover the gap between the MMD, Wasserstein distance and UOT for standard parameter choices. However, it's worth noting that, for instance, the average disparity in Tables~\ref{Tab:DisMMDWas} and~\ref{Tab:DisMMDUOT} diminishes as $ℓ↗∞$ (cf.\ Table~\ref{tab:0}), which typically isn't the most relevant scenario.
\end{remark}

\begin{remark}[Proposition~\ref{prop:547}]
	Based on the experimental observations aligned with the theoretical findings from Proposition~\ref{prop:547} (first assertion), we note that as $η₁= η₂↘ 0$, the Wasserstein distance is recovered, while as $η₁= η₂↗ ∞$, the disparity between $\uot_{r;η}(μ,ν)$ and $u⋅ w_{r}(P,Q) + η_1 Ⅾ(P‖ μ)+ η_2 Ⅾ(Q‖ ν)$ increases.
	However, it is noteworthy that the experiment for both the assertion confirms the theoretical underpinnings consistently across varying parameters.
\end{remark}
\end{revise}
\subsection{Existing approaches: a comparative exploration}
The subsequent discussion is a comparative exploration of state-of-the-art algorithms and other prominent fast algorithms with our NFFT accelerated Sinkhorn’s UOT and MMD.

\subsubsection{State-of-the-art algorithms}
Existing literature highlights remarkable achievements in solving transportation problems using advancements in standard optimization techniques (cf.\ \citet{gottschlich2014shortlist}, \citet{schmitzer2016sparse}).
For concise implementations of such approaches, we reference to the R package \citet{RTransport}.

\begin{table}[htb]
	\footnotesize
	\begin{centering}
		\begin{tabular}{lcccc}
			\toprule
			$n=\tilde n$: &    \multicolumn{2}{c}{\num{50}} & \multicolumn{2}{c}{\num{100}}\tabularnewline
			\midrule Dataset: synthetic data & time & memory & time & memory \tabularnewline & (sec) & (MB)& (sec) & (MB)\tabularnewline
			\midrule
			Network flow
			& \textbf{\num{0.98}}& \num{164.73}& \num{7.08}& \num{3509}\tabularnewline
			Revised simplex
					& \num{1.84} & \num{221.67} & \num{44.68} & \num{150.28}
					\tabularnewline
					Algorithm~\ref{alg:logSinkhornFFT_UOT}
					& \num{1.19} & \textbf{\num{0.44}} & \textbf{\num{3.43}} & \textbf{\num{0.61}}
					\tabularnewline[2\doublerulesep]
			\bottomrule
		\end{tabular}
	\par\end{centering}
\smallskip
	\caption{\label{tab:RPackage}Performance analysis.
	The \emph{network flow} exhibits faster computation for smaller problem instances; for larger problem instances, our algorithm demonstrates notable supremacy in memory and time allocation. Moreover, our algorithm can handle larger scale problems, which are unfeasible for standard and traditional methods.
	The parameters are $r=2$, $λ=20$ (only for Algorithm~\ref{alg:logSinkhornFFT_UOT}) and $η= 1$; the best outcomes are in bold.}
\end{table}

Table~\ref{tab:RPackage} presents a comparative analysis of time and memory performance between the \emph{unbalanced} function from the R package and our proposed Algorithm~\ref{alg:logSinkhornFFT_UOT}.

The results demonstrate the computational superiority of our approach.
The state-of-the-art algorithms from the R package, employing the \emph{network flow} algorithm and \emph{revised simplex algorithm} methods, are designed for solving the standard UOT problem, which do not require regularization.
This aspect can be advantageous when working with small-scale or intermediate-scale problems.

The computational aspects of MMD have not been studied as extensively as those of UOT.
The forthcoming discussion theoretically compares the prominent approaches with our proposed approach based on NFFT.

\subsubsection{Nyström approximation}
In terms of computational methodologies, the utilization of Nyström approximation emerges as imperative, especially in the context of approximating large matrices within diverse machine learning and data analysis applications (cf.\ \citet{NysVsRandFour}).
The Nyström method builds on a low rank approximation of the kernel matrix.
This method gain prominence when prioritizing computational efficiency is paramount.
Its applicability extends to the realms of regularized OT and MMD, where it yields good results (cf.\ \citet{NEURIPS2019_f55cadb9}, \citet{cherfaoui2022discrete}).
Nevertheless, the Nyström approximation approach within the framework of regularized OT and MMD is not without limitations, which include the following:
\begin{description}
 	\item[Loss of spectral information.] Nyström approximation involves selecting a subset of data points to construct a low-rank approximation. This process may lead to a loss of spectral information in the kernel matrix.
 	\item[Computational complexity.] It still involves matrix inversion and multiplication, which can be computationally expensive for large datasets.
 	\item[Sensitivity to subset selection.] The performance of the Nyström approximation is sensitive to the choice of the subset of data points used for the low-rank approximation.
		In certain cases, a suboptimal selection of these points results in a less accurate representation of the original kernel matrix.
 	\item[Applicability to non-uniform distributions.] The method may not perform optimally for datasets with non-uniformly distributed points. 
\end{description}
\subsubsection{Random Fourier features}
On the other hand, for faster computation, the random Fourier features~(RFF) is considered in the context of MMD (cf.\ \citet{zhao2015fastmmd}), where it shows better performance.

However, the RFF is also not without limitations in the MMD setup:
\begin{description}
	\item[Loss of spatial structure.] RFF-based approximations may not fully capture the spatial structure of the original data, especially in MMD, where the spatial arrangement of features is crucial.
	\item[Quality of approximation.] The quality of the approximation introduced by RFF rely on random projections to approximate feature maps, and while they offer computational efficiency for MMD problems, the random nature of the projections may lead to suboptimal approximations.
	\item[Increased computational cost.] RFF introduce additional computational costs, particularly in terms of matrix-vector multiplications, due to the need to compute the random feature mappings for each data point.
	\item[Suboptimal for sparse data.] RFF may not be the most efficient choice for sparse data matrices, as the random features approach might not fully exploit the sparsity inherent in the data during matrix-vector multiplications.
\end{description}

\subsubsection{NFFT based fast summation}
Our methodology for regularized UOT and MMD, leveraging the NFFT, theoretically assures efficient performance even in the limitations mentioned above, see \citet[Chapter~7.2]{plonka2018numerical}.
Some pivotal supremacies are highlighted in following:
\begin{description}
	\item[Efficiency in large-scale computations.]
	The NFFT fast summation technique is known for its efficiency in large-scale computations, making it well-suited for handling datasets with a substantial number of data points in regularized UOT and MMD.
	NFFT can outperform Nyström and RFF when dealing with extensive datasets, as it offers computational advantages in terms of both time and memory.
	\item[Improved accuracy in approximation.] NFFT often provides a more accurate approximation of kernelized functions compared to Nyström and random Fourier features.
	The NFFT contributes to improved accuracy in approximating the feature space, leading to better representation of complex relationships within the data.
	\item[Preservation of structural information.] NFFT is designed to preserve the structural information present in the data, making it particularly advantageous for regularized UOT and MMD, where capturing intricate details in the feature space is crucial.
	Unlike RFF methods, NFFT aims to maintain the essential characteristics of the data during the approximation process.
	\item[Utilization of nonequispaced data.] The adaptability of non-equispaced data by NFFT contributes to its effectiveness, allowing for more flexible and adaptive to wide range of applications.
\end{description}
\begin{revise}
\begin{remark}[Three dimension approximation]
	The adaptation of existing methodologies supporting one and two-dimensional NFFT approximations for standard OT and Multi-Marginal setups is feasible for the three-dimensional NFFT approximation.
\end{remark}
\end{revise}

\section{Summary\label{Sec:summary}}
Our work delivers significant advancements in the realm of comparing unbalanced measures (i.e., non-probability measures).
\rev{We have introduced robust inequalities, which elucidate the nuanced relationship between Wasserstein distance, UOT and MMD (probability and non-probability measures).

}
Furthermore, we explore convergence of regularized UOT to the energy distance (MMD).

Introducing a fast summation technique based on NFFT, our approach enables fast matrix-vector operations, significantly accelerating the implementation of regularized UOT and MMDs while ensuring computational stability, see Table~\ref{tab:summary}.
Diverse kernel options, including inverse multiquadratic and energy kernels, are presented.
\begin{table}[h!t]
	\begin{centering}
	\begin{tabular}{lcc}
		\toprule
		Kernels & standard computation  & NFFT computation  \tabularnewline
		\midrule
		& \multicolumn{2}{c}{Regularized UOT}  \tabularnewline
		 \midrule
		Laplace~($r=1$)
		& $\mathcal{O}\big(\eta\lambda\,(\mu + {\nu})\, d\,n^2 \log n \,\big)$& $\mathcal{O}\big(\eta\lambda\,(\mu + {\nu})\,d\,n\,\log² n \,\big)$ \tabularnewline
		Gaussian~($r=2$)
		& $\mathcal{O}\big(\eta\lambda\,(\mu + {\nu})\, d\,n^2 \log n \,\big)$& $\mathcal{O}\big(\eta\lambda\,(\mu + {\nu})\,d\,n\,(\log n)\,\big)$ \tabularnewline
		\midrule
		&\multicolumn{2}{c}{MMDs}  \tabularnewline
		 \midrule
		Energy
		& $\mathcal{O}(n^2\,\log n\, d)$& $\mathcal{O}(n\,\log n\, d)$ \tabularnewline
		Gaussian
		& $\mathcal{O}(n^2\,\log n\, d)$& $\mathcal{O}(n\, d)$ \tabularnewline
		Laplace
		& $\mathcal{O}(n^2\,\log n\, d)$& $\mathcal{O}(n\,\log n\, d)$
		\tabularnewline
		Inverse multiquadratic
		& $\mathcal{O}(n^2\,\log n\, d)$ & $\mathcal{O}(n\,\log n\, d)$ \tabularnewline
		\bottomrule
	\end{tabular}
\par\end{centering}
\smallskip
\caption{Arithmetic operations of standard computation vs.\ NFFT computation\label{tab:summary}}
\end{table}

\rev{We again emphasize that these inequalities along with fast computation significantly enhance both the theoretical and computational frameworks, enabling effective management of disparities between probability and non-nonnegative (unbalanced) measures.}

Furthermore, we showcase numerical illustrations of our method, validating its robustness, stability, and properties outlined in theoretical results. These numerical demonstrations provide empirical evidence, further substantiating the efficacy of our proposals.


The corresponding implementations are available in the following GitHub repository
\begin{center}
\url{https://github.com/rajmadan96/FastUOTMMD.git} .
\end{center}

\subsection*{Acknowledgement}

We would like to thank Franziska Nestler for helping with the implementation of fast energy kernel.

\subsection*{Ethics declarations}

\subsubsection*{Conflict of interest}
The authors have not disclosed any competing interests.
\appendix

\section{Proofs}
    \subsection{Proof of Proposition~\ref{prop:464}\label{prop:464Proof}}
	Let $π$ be any measure with marginals~\eqref{eq:466} and set $π_c≔ c⋅ π$, where $c≥0$. It follows that
	\begin{align}
		&τ_{1\#} π_c(⋅)= c⋅π(⋅×𝓧)= c⋅P(⋅)= c⋅{μ(⋅) ∕ μ(𝓧)} \text{ and}\\
		&τ_{2\#} π_c(⋅)= c⋅π(𝓧×⋅)= c⋅Ꝓ(⋅)= c⋅{ν(⋅) ∕ ν(𝓧)},
	\end{align} so that the Radon–Nkodým derivatives are constant,
	\begin{align}
		&{ⅾτ_{1\#} π_c ∕ ⅾμ}= {c ∕ μ(𝓧)}\text{ and}\\
		&{ⅾτ_{2\#} π_c ∕ ⅾν}= {c ∕ ν(𝓧)}.
	\end{align}
	It follows with~\eqref{eq:298} that the Kullback–Leibler divergences are
	\begin{align}
		&\KL(τ_{1\#} π_c‖ μ)= {c ∕ μ(𝓧)}\log⟮{c ∕ μ(𝓧)}⟯μ(𝓧) + μ(𝓧)- c \text{ and}\\
		&\KL(τ_{2\#} π_c‖ ν)= {c ∕ ν(𝓧)}\log⟮{c ∕ ν(𝓧)}⟯ν(𝓧) + ν(𝓧)- c.
	\end{align}
	Restricting the problem UOT to particular measures of the form $π_c=c⋅π$, the objective of~\eqref{eq:W3} reduces to
	\begin{align}\label{eq:2670}
		c⋅❬π| d^r❭
			&+ η₁❨{c ∕ μ(𝓧)}\log⟮{c ∕ μ(𝓧)}⟯μ(𝓧) + μ(𝓧)- c❩ \\
			&+ η₂❨{c ∕ ν(𝓧)}\log⟮{c ∕ ν(𝓧)}⟯ν(𝓧) + ν(𝓧)- c❩.
	\end{align}

	Minimizing the objective~\eqref{eq:2670} with respect to the constants $c>0$ reveals the best constant
	\[	c^*_{r;η}= e^{-{❬ π| d^r❭ ∕ η₁+η_2}}⋅ μ(𝓧)^{η₁ ∕ η₁+η₂}⋅ν(𝓧)^{η₂ ∕ η₁+η₂}\]
	and thus the assertion.

	\begin{remark}\label{rem:507}
		For increasing and equal regularization parameters, $η₁= η₂→ ∞$, the measure
		\[	{μ⊗ ν ∕ √{μ(𝓧)⋅ν(𝓧)}} \]
		constitutes an upper bound.
		Further, for $η₂$ constant but $η₁→∞$, say, the measure in~\eqref{eq:478} approaches $μ(𝓧)⋅π$: this measure has marginal~$μ$. Apparently, this is in line with the objective in~\eqref{eq:W3}, which forces the first marginal to approach~$μ$.
	\end{remark}

	\subsection{Proof of Proposition~\ref{prop:536}\label{prop:536_proof}}
	Observe first that the marginal measures are
	\begin{align}
		&τ_{1\#} π_c(⋅)= π_c(⋅×𝓧)= c μ(⋅) ν(𝓧)\text{ and}\\
		&τ_{2\#} π_c(⋅)= π_c(𝓧×⋅)= c μ(𝓧) ν(⋅),
	\end{align} so that the Radon–Nikodým derivatives are constant,
	\begin{equation}
		{ⅾπ_c ∕ ⅾμ⊗ν}= c, 
		{ⅾτ_{1\#} π_c ∕ ⅾμ}= c ν(𝓧) \text{ and } 
		{ⅾτ_{2\#} π_c ∕ ⅾν}= c μ(𝓧).
	\end{equation}
	It follows with~\eqref{eq:298} that
	\begin{align}
		&\KL(π_c‖ π)= c \log(c)⋅μ(𝓧)ν(𝓧)+ μ(𝓧)ν(𝓧)- c μ(𝓧)ν(𝓧),
	\intertext{and the Kullback–Leibler divergences of the marginals thus are}
		&\KL(τ_{1\#} π_c‖ μ)= c ν(𝓧)\log❪c ν(𝓧)❫μ(𝓧) + μ(𝓧)- c μ(𝓧)ν(𝓧) \text{ and}\\
		&\KL(τ_{2\#} π_c‖ ν)= c μ(𝓧)\log❪c μ(𝓧)❫ν(𝓧) + ν(𝓧)- c μ(𝓧)ν(𝓧).
	\end{align}
	Restricting the problem UOT to particular measures of the form $π_c=c⋅μ⊗ ν$, the objective of~\eqref{eq:W4} reduces to
	\begin{align}\label{eq:1117}
		c⋅❬μ⊗ ν| d❭
			&+ {1∕λ} ❨c \log(c)⋅μ(𝓧)ν(𝓧) + c μ(𝓧)ν(𝓧)- μ(𝓧)ν(𝓧)❩ \\
			&+ η₁❨c ν(𝓧)\log❪c ν(𝓧)❫μ(𝓧) + c μ(𝓧)ν(𝓧)- μ(𝓧)❩ \\
			&+ η₂❨c μ(𝓧)\log❪c μ(𝓧)❫ν(𝓧) + c μ(𝓧)ν(𝓧)- ν(𝓧)❩.
	\end{align}

	Minimizing the objective~\eqref{eq:1117} with respect to the constants $c>0$ reveals the best measure explicitly as
	\[	π⃰ ≔ c^*_{r;η;λ} ⋅ μ⊗ ν,	\] where $c⃰ $ as given in~\eqref{eq:561}
	is the best possible constant.
	\rev{\subsection{Proof of Corollary~\ref{cor:477}\label{cor:477Proof}}}
	Consider the debiased regularized UOT formulation~\eqref{eq:586}
	\begin{align}\label{eq:thm1}
		\mathrm{sd}_{1;η;\lambda}(P,\tilde{P}) = & \,{\uot}_{1;\eta;\lambda}({P},\tilde{P}) - \frac{1}{2}\,{\uot}_{1;\eta;\lambda}(P,{P}) - \frac{1}{2}\, {\uot}_{1;\eta;\lambda}(\tilde{P},\tilde{P}) \\ & + \frac{1}{2\,\lambda}\big(P(\mathcal{X}) - \tilde{P}(\mathcal{X})\big)^2
	\end{align} and note that the last component  vanishes as $P(\mathcal{X}) = \tilde{P}(\mathcal{X}) =1$.
	Thus, as ${\uot}_{1;\eta;\lambda}(\cdot, \cdot) \to {\uot}_{1;\eta}(\cdot, \cdot)$ for $\lambda \to \infty$, we have that
	\begin{align*} \lim_{\lambda \to \infty} \mathrm{sd}_{1;\eta;\lambda}(P,\tilde{P}) = {\uot}_{1;\eta}({P},\tilde{P}) - \frac{1}{2}\,{\uot}_{1;\eta}(P,{P}) - \frac{1}{2}\, {\uot}_{1;\eta}(\tilde{P},\tilde{P}) = {\uot}_{1;\eta}({P},\tilde{P}),
	\end{align*}
	which is the first assertion.

	For the second assertion, we observe from Definition~\ref{def:526} of ${\uot}_{1;η;λ}(P,Ꝓ)$ that
	\begin{align*}
		\lim_{η → ∞} {\uot}_{1;η;λ}(P,Ꝓ)= w_{1;λ}(P,Ꝓ),
			\end{align*}
		where ${w}_{1,λ}(P,Ꝓ)$ is defined as in~\eqref{eq:sd_OT}. Similarly, we obtain
		\begin{align*}
			\lim_{η→∞} \mathrm{sd}_{1;η;λ}(P,Ꝓ) = \mathrm{sd}_{1;λ}({P},Ꝓ)
		\end{align*}
		for the associated Sinkhorn divergences, and hence
		\begin{align*}
			\mmd_{k^{\text{e}}}(P,Ꝓ)
			= 	\lim_{λ → 0} \mathrm{sd}_{1;λ}(P,Ꝓ)
			= \lim_{λ → 0} \lim_{η→∞} \mathrm{sd}_{1;η;λ}(P,Ꝓ)
		\end{align*}
		by Lemma~\ref{lem:inter_ot_mmd}.
		This is the second assertion, which concludes the proof.
\bibliographystyle{abbrvnat}
\bibliography{LiteraturAlois,LiteraturRaj}
\end{document}